\documentclass[12pt]{amsart}
\usepackage{amsmath}
\usepackage{amsfonts}
\usepackage{amsthm}
\usepackage{amssymb}
\usepackage{tikz-cd}
\usepackage{graphics}
\usepackage{array}
\usepackage{dsfont}
\usepackage{color}
\usepackage{wasysym}
\usepackage{hyperref}
\usepackage{pdfsync}
\usepackage{mathrsfs}
\usepackage[alphabetic]{amsrefs}
\usepackage{geometry}
\usepackage{bbm}
\usepackage[all,cmtip]{xy}
\usepackage{tikz}\usetikzlibrary{intersections}
\newcommand{\R}{{\mathbb R}}

\newcommand{\Q}{{\mathbb Q}}

\DeclareMathOperator{\Supp}{Supp}

\DeclareMathOperator{\Exc}{Exc}

\DeclareMathOperator{\red}{red}

\theoremstyle{plain}

\newtheorem{thm}{Theorem}[section]

\newtheorem{prop}[thm]{Proposition}
\newtheorem{coro}[thm]{Corollary}
\newtheorem{lemma}[thm]{Lemma}

\theoremstyle{definition}

\newtheorem{defi}[thm]{Definition}
\newtheorem{exa}[thm]{Example}
\newtheorem{rem}[thm]{Remark}

\newcommand{\bP}{\mathbb{P}}

\newcommand{\bQ}{\mathbb{Q}}

\newcommand{\bN}{\mathbb{N}}

\newcommand{\oX}{\overline{X}}
\newcommand{\oY}{\overline{Y}}
\newcommand{\oZ}{\overline{Z}}
\newcommand{\oB}{\overline{B}}
\newcommand{\oD}{\overline{D}}
\newcommand{\oV}{\overline{V}}

\newcommand{\oDelta}{\overline{\Delta}}
\newcommand{\oGamma}{\overline{\Gamma}}
\newcommand{\oPhi}{\overline{\Phi}}

\newcommand{\tX}{\widetilde{X}}

\newcommand{\tZ}{\widetilde{Z}}

\newcommand{\tD}{\widetilde{D}}
\newcommand{\tB}{\widetilde{B}}

\newcommand{\tH}{\widetilde{H}}
\newcommand{\tf}{\widetilde{f}}
\newcommand{\tvarphi}{\widetilde{\varphi}}
\newcommand{\tpsi}{\widetilde{\psi}}
\newcommand{\tDelta}{\widetilde{\Delta}}

\newcommand{\cX}{\mathcal{X}}

\newcommand{\cO}{\mathcal{O}}

\newcommand{\cM}{\mathcal{M}}

\newcommand{\cD}{\mathcal{D}}
\newcommand{\cP}{\mathcal{P}}

\newcommand{\cK}{\mathcal{K}}
\newcommand{\cN}{\mathcal{N}}
\newcommand{\cQ}{\mathcal{Q}}

\newcommand{\va}{\mathbf{a}}
\newcommand{\vb}{\mathbf{b}}
\newcommand{\vc}{\mathbf{c}}

\newcommand{\vv}{\mathbf{v}}

\DeclareMathOperator{\Diff}{Diff}

\DeclareMathOperator{\Sing}{Sing}
\DeclareMathOperator{\lc}{lc}

\newcommand{\norm}[1]{\left\lVert#1\right\rVert}
\newcommand{\simover}[1]{\sim_{\R,#1}}

\numberwithin{equation}{section}

\begin{document}

\title[Wall crossing]{MMP for locally stable families and wall crossing for moduli of stable pairs}
	
	\author{Fanjun Meng}
	\address{Department of Mathematics, Johns Hopkins University, 3400 N. Charles Street, Baltimore, MD 21218, USA}
	\email{fmeng3@jhu.edu}
	
	\author{Ziquan Zhuang}
	\address{Department of Mathematics, Johns Hopkins University, 3400 N. Charles Street, Baltimore, MD 21218, USA}
	\email{zzhuang@jhu.edu}


\begin{abstract}
We construct reduction and wall-crossing morphisms between the moduli spaces of stable pairs as the coefficients vary, generalizing the earlier work of Ascher, Bejleri, Inchiostro and Patakfalvi which deals with the klt case. Along the proof, we show that one can run the MMP with scaling on normal locally stable families over a normal base, and that the existence of good minimal models is preserved when reducing coefficients away from zero.
\end{abstract}

	\maketitle
	% \tableofcontents

\section{Introduction}

In \cite{Has03}, Hassett shows that there are natural reduction and wall-crossing morphisms between the moduli spaces of weighted pointed stable curves as the weights vary. The higher-dimensional generalization of weighted pointed stable curves, whose study goes back to \cite{KSB88}, is the notion of stable pairs. These are semi-log-canonical (slc) pairs $(X,D=\sum a_i D_i)$ (the basic examples are simple normal crossing pairs $(X,D)$ with coefficients $a_i\in (0,1]$) such that $K_X+D$ is ample. The moduli theory of stable pairs has been established in the past several decades; we refer to \cite{Kol23} for a comprehensive discussion. In particular, for any dimension $d$, any coefficient vector $\va$ for the boundary divisor $D$, and any volume $v=(K_X+D)^d >0$, there exists a proper moduli space that parameterizes stable pairs with the given numerical invariants.

It is natural to compare these moduli spaces as the coefficients vary, and one may ask whether the reduction and wall-crossing morphisms generalize to the moduli of stable pairs. Such morphisms play an important role in the study of explicit moduli compactifications, and have been worked out in several explicit examples, see e.g. \cites{Ale15,Inc20,AB21}. The general question is taken up recently in \cite{ABIP23} under a klt assumption. The goal of this article is to remove this klt condition and to construct, in the general setting, the reduction and wall-crossing morphisms between the moduli spaces of stable pairs.

These morphisms are typically obtained through an ample model construction: if we denote by $\cK_\va$ the moduli space of stable pairs with coefficient vector $\va$, and let $\vb$ be another coefficient vector, then the expected natural morphism $\cK_\va \to \cK_\vb$ should send a stable pair $(X,\sum a_i D_i)$ to the semi-log-canonical model of $(X,\sum b_i D_i)$, which is roughly the ample model of $K_X+\sum b_i D_i$. Since semi-log-canonical models of non-normal semi-log-canonical pairs do not always exist (see e.g. \cites{Kol11-slc-no-MMP,AK19} or Section \ref{ss:slc counterexample} for some examples), this imposes restrictions on the locus where the morphism $\cK_\va \to \cK_\vb$ can be defined. Another related restriction, observed in \cite{ABIP23}, is that the reduction morphisms can only be defined on the normalization of the moduli spaces. With these in mind, our main results give some optimal statements that one may hope for. Roughly speaking, we show that wall-crossing morphisms exist on the locus that parameterizes degenerations of normal (hence log canonical) stable pairs, while the reduction morphisms exist on the normalization of the same locus.

\subsection{Main theorems}

To state our main theorems more precisely, we need to introduce some notation. Fix a locally stable family (see Definition \ref{defi:locally stable families of pairs}) $f\colon (X,\sum_{i=1}^n a_i D_i)\to B$ over a normal (but not necessarily proper) variety. A special case is when $(X,\sum D_i)$ is log smooth over $B$ and $0<a_i\le 1$ for all $i$. For ease of notation, we write $\va=(a_1,\dots,a_n)$. For any coefficient vector $\vb=(b_1,\dots,b_n)$, we denote by $\vb D$ the divisor $\sum_{i=1}^n b_i D_i$. Following \cites{Has03, ABIP23}, we say that a coefficient vector $\vb=(b_1,\dots,b_n)$ is admissible if $0<b_i\le a_i$ for all $i$ and $K_{X/B}+\vb D$ is big over $B$. 

\begin{thm}[reduction] \label{main thm:weight reduction}
Assume that the generic fiber of $(X, \va D)\to B$ is normal and has a log canonical model. Then for any admissible coefficient vector $\vb$, there exists a $($unique$)$ stable family $(Z_\vb,\vb \Delta_\vb)\to B$ that is the log canonical model of $(X,\vb D)\to B$ $($i.e. the ample model of $K_{X/B}+\vb D$ over $B)$. In particular, it induces a unique morphism $\Phi_\vb \colon B\to \cK_\vb$.
\end{thm}

In particular, if we denote by $\overline{\cK}_\va^{\lc}$ the normalized closure of the locus in the moduli space $\cK_\va$ that parameterizes log canonical stable pairs with coefficient vector $\va$, then the above theorem gives the reduction morphism $\overline{\cK}_\va^{\lc}\to \overline{\cK}_\vb^{\lc}$ for every admissible coefficient vector $\vb$. When the generic fiber of $(X, \va D)\to B$ is klt and the base $B$ is smooth (this is one of the cases considered in \cite{ABIP23}), Theorem \ref{main thm:weight reduction} follows from \cite{BCHM10}.

We say that a polytope $\cP\subseteq \R^n$ is admissible if every coefficient vector in $\cP$ is admissible. Under the assumptions of Theorem \ref{main thm:weight reduction}, as the coefficient $\vb$ varies in an admissible polytope $\cP$, we also have a wall-and-chamber decomposition of $\cP$ corresponding to the reduced closure $\cM_\vb := \overline{\Phi_\vb(B)}$ of the images of the reduction morphisms $\Phi_\vb\colon B\to \cK_\vb$.

\begin{thm}[wall crossing] \label{main thm:wall crossing}
Under the assumptions of Theorem \ref{main thm:weight reduction}, let $\cP$ be an admissible rational polytope. Then there exists a finite rational polyhedral decomposition $\cP=\cup_{i=1}^p \cP_i$ such that:
\begin{enumerate}
    \item In any open chamber $\cP_i^\circ$, the family $(Z_\vb, \Delta_\vb)\to B$ of marked pairs is independent of $\vb\in \cP_i^\circ$, and $\cM_\vb\cong \cM_{\vb'}$ for any $\vb,\vb'\in \cP_i^\circ$.
    
    \item For any $i$, we denote $\cM_\vb$ by $\cM_i$ for any $\vb\in \cP_i^\circ$ $($this is well-defined by part $(1))$. Then for any $i,j$ such that $\cP_j$ is a face of $\cP_i$, there exists a $($unique$)$ morphism $\alpha_{ij}\colon \cM_i\to \cM_j$ making the following diagram commute.
    \[
    \xymatrix{
		B \ar[dr] \ar[d] & \\
		\cM_i \ar[r]^{\alpha_{ij}}  & \cM_j 
        }
    \]
\end{enumerate}
\end{thm}

The morphisms $\alpha_{ij}$ are what we refer to as the wall-crossing morphisms. Moreover, the morphism $\alpha_{ij}$ lifts to a morphism between the universal families over $\cM_i$ and $\cM_j$. Putting together Theorems \ref{main thm:weight reduction} and \ref{main thm:wall crossing}, we recover Hassett's result in the curve case. When the generic fiber of $(X,\va D)\to B$ is klt, the above theorem also improves upon \cite{ABIP23}*{Theorem 1.8}, where the wall-crossing morphisms are constructed on the seminormalization of the corresponding moduli locus.

There are also reduction and wall-crossing morphisms in the semi-log-canonical case if we impose additional assumptions. The following is an example. It is a generalization of \cite{ABIP23}*{Theorem 1.1} to the semi-log-canonical setting. For any $\va=(a_1,\ldots,a_n)$ and $\vb=(b_1,\ldots,b_n)$, we write $\vb\le\va$ if $b_i\le a_i$ for all $i$.

\begin{thm} \label{main thm:slc case}
Let $(X,\sum D_i)\to B$ be a family of marked pairs over a normal variety and let $\cP$ be a rational polytope such that $(X,\va D)\to B$ is a stable family for every $\va\in \cP$. Let $\cN_\va$ be the normalized closure of the image of the morphism $\Phi_\va\colon B\to \cK_\va$. Then there is a canonical birational reduction morphism $r_{\va ,\vb}\colon \cN_\va \to \cN_\vb$ for any $\va \ge \vb$ in $\cP$. Moreover, we have $r_{\vb,\vc}\circ r_{\va,\vb} = r_{\va,\vc}$ for any $\va \ge \vb \ge \vc$ in $\cP$, and there is a wall-and-chamber and wall-crossing structure on $\cP$ corresponding to $\cN_\va$ as in Theorem \ref{main thm:wall crossing}.
\end{thm}

\subsection{Other results}

The proof of the main theorems relies on several other results that might be of independent interest. Some of these results are already known in the klt case, but the extension to the log canonical case is more subtle.

A famous conjecture in birational geometry is that every log canonical pair of nonnegative Kodaira dimension has a good minimal model (Definition \ref{defi:good minimal model}). For klt pairs of log general type this follows from \cite{BCHM10}, but if the pair is only log canonical then the conjecture is wide open even in the log general type case. This presents one of the difficulties in the construction of the reduction morphisms in the log canonical setting. The first ingredient that goes into our proof of the main theorems is the existence of good minimal models for those special log canonical pairs that come up in Theorem \ref{main thm:weight reduction}. In fact we prove a more general statement that abundance is preserved when reducing coefficients away from zero.

\begin{thm}[=Theorem \ref{thm:gmm smaller coef}] \label{main thm:gmm smaller coef}
Let $(X,D)$ be a log canonical pair over $B$ and let $0\le \Delta\le D$ be an $\R$-divisor. Assume that $(X,D)$ has a good minimal model over $B$, $K_X+\Delta$ is $\R$-Cartier, and $\Supp(\Delta)=\Supp(D)$. Then $(X,\Delta)$ has a good minimal model or a Mori fiber space over $B$.
\end{thm}

This implies Theorem \ref{main thm:weight reduction} when the base $B$ is a single point. In general, we need to show that the log canonical model construction works in locally stable families. Thus the next ingredient is a minimal model program (MMP) for such families.

\begin{thm}[=Theorem \ref{thm:MMP locally stable family}] \label{main thm:MMP locally stable family}
Let $f\colon (X,D)\to B$ be a locally stable family over a normal variety with normal generic fiber. Then for any $f$-ample $\R$-divisor $H$ on $X$, we may run the $(K_{X/B}+D)$-MMP over $B$ with scaling of $H$. Moreover, the MMP terminates with a good minimal model over $B$ if the generic fiber has a good minimal model.
\end{thm}

Although MMP with scaling for log canonical pairs is known by \cites{Amb03,Fuj11,Bir12-lc-flip,HX13,HH20}, the singularities of the pair $(X,D)$ in the above theorem can be much worse, because the base $B$ may have very bad singularities. On the other hand, $(X,D)$ can be made log canonical after a base change. The above theorem is therefore a descent statement: in particular, we show that extremal contractions (resp. flips) descend along arbitrary reduced (resp. normal) base change. Ultimately, these rely on, among other things, the invariance of plurigenera for good minimal models of locally stable families. We remark that when the generic fiber of the family is klt, some special cases of Theorem \ref{main thm:MMP locally stable family} are implicit in \cite{ABIP23} when they analyze the invertibility of certain special wall-crossing morphisms. Using Theorem \ref{main thm:MMP locally stable family} we are thus able to give a more direct and intrinsic construction of the reduction morphisms (that is, without analyzing the aforementioned wall-crossing morphisms) even in the klt case.

Finally, to obtain the wall-and-chamber structure in Theorem \ref{main thm:wall crossing}, the third ingredient is the finiteness of ample models as the coefficients vary, see Theorem \ref{thm:chamber decomp by ample model}. For klt pairs of log general type, this again follows from \cite{BCHM10}. We refer to Section \ref{sec:Finiteness of ample models} for the relevant setup and the precise statement. 

\subsection{Outline}

This paper is organized as follows. Section \ref{sec:preliminaries} covers some preliminary definitions and statements. Section \ref{sec:MMP results} studies the existence of good minimal models for (some) log canonical pairs, Section \ref{sec:MMP locally stable} establishes the MMP with scaling for locally stable families, while Section \ref{sec:Finiteness of ample models} deals with the finiteness of ample models. The main theorems are proved in Section \ref{sec:wall crossing}, and Section \ref{sec:examples} collects some examples to illustrate why some of the assumptions in our main results are optimal.

\subsection*{Acknowledgements}

The authors would like to thank Dori Bejleri, Giovanni Inchiostro, Vyacheslav Shokurov and Chenyang Xu for helpful discussions and comments. The second author is partially supported by a Clay research fellowship, a Sloan fellowship, and the NSF (DMS-2240926, DMS-2234736). 

\section{Preliminaries} \label{sec:preliminaries}

We work over an algebraically closed field $k$ of characteristic $0$. All schemes are of finite type over $k$, unless otherwise stated.

\subsection{Notation and conventions}

If $f\colon Y\dashrightarrow X$ is a birational map and $D$ is an $\R$-divisor on $X$, we denote its strict transform $f_*^{-1} D$ on $Y$ by $D_Y$. If $D$ is $\R$-Cartier, its \emph{birational pullback} $f^*D$ is defined as $p_*q^*D$ where $p\colon Z\to Y$, $q\colon Z\to X$ is a common resolution (the definition is independent of the choice of the common resolutions).

If $X\to B$ is a projective morphism of normal quasi-projective varieties and $D\ge 0$ is an $\R$-divisor on $X$, we call $(X,D)$ a \emph{pair over $B$}.

\subsection{Singularities of pairs}

For basic definitions of singularities of pairs, we refer to \cite{KM98} and \cite{Kol13}. Occasionally we make use of generalized pairs, and we refer to \cite{BZ16}*{Section 4} for their definition and basic properties. In this subsection, we recall several definitions regarding stable pairs.

\begin{defi}
A scheme $X$ is called \emph{demi-normal} if it is $S_2$ and its codimension $1$ points are either regular points or nodes.
\end{defi}

\begin{defi}
Let $\pi\colon\bar{X}\to X$ be the normalization of a demi-normal scheme $X$. The \emph{conductor ideal}
$$\mathcal{H}om_{\cO_{X}}(\pi_*\cO_{\bar{X}},\cO_X)\subset\cO_X$$
defines reduced, pure codimension $1$ and closed subschemes $\Gamma\subset X$ and $\bar{\Gamma}\subset\bar{X}$ which are called the \emph{conductors} or the \emph{double loci}. 
\end{defi}

\begin{defi}
Let $X$ be a demi-normal scheme with normalization $\pi\colon\bar{X}\to X$ and conductors $\Gamma\subset X$ and $\bar{\Gamma}\subset\bar{X}$. Let $D$ be an effective $\R$-divisor on $X$ whose support does not contain any irreducible component of $\Gamma$ and $\bar{D}$ the divisorial part of $\pi^{-1}(D)$. The pair $(X,D)$ is called \emph{semi-log-canonical} (or \emph{slc}) if
\begin{enumerate}	
    \item[(1)] $K_X+D$ is $\R$-Cartier, and

    \item[(2)] $(\bar{X},\bar{D}+\bar{\Gamma})$ is log canonical.
\end{enumerate}
Moreover, if $X$ is proper and $K_X+D$ is ample, $(X,D)$ is called a \emph{stable} pair.
\end{defi}

\subsection{Models}

We need to deal with different kinds of models of log pairs for our purpose. In this subsection, we recall their definitions.

\begin{defi}\label{bigness}
Let $f\colon X\to B$ be a projective morphism of quasi-projective varieties. Let $D$ be an $\R$-divisor on $X$ such that $X$ is smooth at the generic points of $D$. We say $D$ is \emph{big over $B$} or \emph{$f$-big} if
$$\limsup\frac{h^0(F,\cO_F(\llcorner mD\lrcorner) )}{m^{\dim F}}>0$$
for the generic fiber $F$ of $f$. Equivalently, $D$ is big over $B$ if $D\sim_{\R, B}A+E$, where $A$ is ample over $B$ and $E\ge0$.
\end{defi}

First we recall the existence of dlt modifications which is useful.

\begin{thm}[\cite{KK10}*{Theorem 3.1},\cite{Fuj11}*{Theorem 10.5}]\label{dlt modification}
Let $(X,D)$ be a log canonical pair where $X$ is a normal quasi-projective variety. Then there exists a projective birational morphism $f\colon Y\to X$ from a normal quasi-projective variety $Y$ such that
\begin{enumerate}
	\item[(1)] $Y$ is $\Q$-factorial.
	
	\item[(2)] Let $E$ be the sum of all the $f$-exceptional divisors. Then $(Y, D_Y +E)$ is dlt and $K_Y+D_Y+E=f^*(K_X+D)$.
\end{enumerate}
\end{thm}

The morphism $f$ in the above theorem is called a \emph{dlt modification} of $(X,D)$, and $(Y, D_Y+E)$ is called a \emph{dlt model} of $(X,D)$.

Next, we recall the definition of good minimal models and log canonical models.

\begin{defi} \label{defi:good minimal model}
Let $(X,D)$ be a log canonical pair over $B$. Suppose that $\pi\colon X\dashrightarrow Y$ is a birational contraction over $B$. Then $(Y,D_Y)$ is called a \emph{minimal model of $(X,D)$ over $B$} if $K_Y+D_Y$ is $\R$-Cartier and nef over $B$, and 
\[a(E,X,D)<a(E,Y,D_Y)\]
for every $\pi$-exceptional divisor $E$ on $X$. It is called a \emph{good minimal model of $(X,D)$ over $B$} if $K_Y+D_Y$ is semiample over $B$. If in addition $K_Y+D_Y$ is also big over $B$, the ample model of $K_Y+D_Y$ over $B$ is called the \emph{log canonical model of $(X,D)$ over $B$}.
\end{defi}

For ease of notation, we will sometimes simply say $Y$ is a minimal (resp. log canonical) model of $(X,D)$ over $B$, as the boundary divisor $D_Y$ is uniquely determined as the strict transform of $D$.

\begin{rem} \label{rem:gmm indep of model}
If in the above definition we only require that $a(E,X,D)\le a(E,Y,D_Y)$, then the model $(Y,D_Y)$ is called a \emph{weak log canonical model of $(X,D)$ over $B$}. Every minimal model is clearly a weak log canonical model. The log canonical model, if exists, is also a weak log canonical model (as the name suggests). Any two weak log canonical models are $K$-equivalent (i.e. the pullbacks of $K_Y+D_Y$ are the same on any common resolution of the models, see e.g. \cite{KM98}*{Proof of Theorem 3.52}), thus the semiampleness of $K_Y+D_Y$ over $B$ is independent of the weak log canonical model. In particular, if $(X,D)$ has a good minimal model over $B$, then every minimal model of $(X,D)$ over $B$ is good, and the log canonical model is unique, c.f. \cite{KM98}*{Corollary 3.53}, \cite{Bir12-lc-flip}*{Remark 2.7}, \cite{HX13}*{Lemma 2.4}, \cite{LM21}*{Lemma 2.6}.
\end{rem}

\begin{rem}
There are several different definitions of minimal models in the literature. When $(X,D)$ is dlt, the above definition coincides with \cite{KM98}*{Definition 3.50}. Many authors (e.g. \cite{HX13}*{p.165}) further require that the minimal models be $\bQ$-factorial. In many applications this is automatic since the MMP starts with a $\bQ$-factorial dlt pair, but in order to run the MMP for locally stable families in Section \ref{sec:MMP locally stable} we need to drop the $\bQ$-factoriality assumption (c.f. \cite{Fuj11} and \cite{Kol21-MMP-non-Q-factorial}). In \cite{Bir12-lc-flip} and \cite{HH20}, the authors consider ($\bQ$-factorial dlt) minimal models $Y$ over $B$ for which the birational map $X\dashrightarrow Y$ is not necessarily a contraction. By \cite{HH20}*{Theorem 1.7}, if $(X,D)$ has a minimal model over $B$ under their definition, then it also has a minimal model over $B$ under our definition. On the other hand, if $(X,D)$ has a minimal model $(Y,D_Y)$ over $B$ under our definition, then any dlt modification of $(Y,D_Y)$ is a minimal model of $(X,D)$ over $B$ under their definition. The minimal models under these various definitions are all $K$-equivalent, hence if one of them is good then so are the others (as in the previous remark). Since we are mainly interested in the existence of (good) minimal models and log canonical models, these different definitions are essentially equivalent for our purpose.
\end{rem}

\begin{lemma} \label{lem:compare gmm existence on bir model}
Let $(X,D)$ and $(Y,\Gamma)$ be log canonical pairs over $B$, and $\pi\colon Y\dashrightarrow X$ a birational contraction over $B$ such that
\begin{equation} \label{eq:K_Y+D_Y=K_X+D+exceptional}
    K_Y+\Gamma=\pi^*(K_X+D)+E
\end{equation}
for an effective $\pi$-exceptional divisor $E$. Then $(Y,\Gamma)$ has a good minimal model (resp. log canonical model) over $B$ if and only if $(X,D)$ does, and the two pairs have the same log canonical model over $B$.
\end{lemma}

In particular, if $Y\dashrightarrow X$ is a $(K_Y+\Gamma)$-MMP over $B$ then $(Y,\Gamma)$ has a good minimal model over $B$ if and only if $(X,\Gamma_X)$ has a good minimal model over $B$, c.f. \cite{Has19}*{Remark 2.16}. The same also applies if $(X,D)$ is a weak log canonical model of $(Y,\Gamma)$ over $B$.

\begin{proof}
This is well-known to experts. The assumption \eqref{eq:K_Y+D_Y=K_X+D+exceptional} implies that $K_X+D$ is big over $B$ if and only if $K_Y+\Gamma$ is big over $B$. Hence if $\pi$ is a morphism, the statement follows from \cite{KM98}*{Corollary 3.53} and \cite{Has19}*{Lemma 2.15} (see also \cite{HX13}*{Lemma 2.10} and \cite{Bir12-lc-flip}*{Remark 2.8}). In the general case, choose a common log resolution
\[
\xymatrix{
& Z \ar[ld]_\varphi \ar[rd]^\psi & \\
Y \ar@{-->}[rr]^\pi & & X\\
}
\]
and consider the divisor $\Delta=\Gamma_Z+F$ on $Z$, where $F$ is the sum of the $\varphi$-exceptional divisors. Since $(Y,\Gamma)$ is log canonical, we have
\[K_Z+\Delta = \varphi^*(K_Y+\Gamma)+E_1\]
for some $\varphi$-exceptional divisor $E_1\ge 0$. By \eqref{eq:K_Y+D_Y=K_X+D+exceptional}, this also implies that
\[K_Z+\Delta = \psi^*(K_X+D)+E_2\]
for some $\psi$-exceptional divisor $E_2\ge 0$. By \cite{Has19}*{Lemma 2.15}, we see that $(Y,\Gamma)$ (resp. $(X,D)$) has a good minimal model over $B$ if and only if $(Z,\Delta)$ has a good minimal model over $B$, and all three pairs have the same log canonical model over $B$ by \cite{KM98}*{Corollary 3.53}.
\end{proof}

\subsection{MMP for log canonical pairs}

We need two fundamental results on the MMP for log canonical pairs. The first result is the existence of good minimal models for pairs of log Calabi-Yau type.

\begin{thm}[\cite{Bir12-lc-flip}*{Theorem 1.1}, \cite{HX13}*{Theorem 1.6}, \cite{Has19}*{Theorem 1.1}] \label{thm:birkar gmm}
Let $(X,D+A)$ be a log canonical pair over $B$ where $D,A\ge 0$ are $\R$-divisors. Assume that
\[K_X+D+A\sim_{\R, B} 0\]
and that $K_X+D$ is $\R$-Cartier. Then $(X,D)$ has a good minimal model or a Mori fiber space over $B$.
\end{thm}

The second result is related to the MMP with scaling. In fact, by taking $X=Y$ in the statement below it implies that any log canonical pair $(X,D+A)$ over $B$ where $A$ is ample over $B$ has a good minimal model over $B$ as long as $K_X+D+A$ is pseudo-effective over $B$.

\begin{thm}[\cite{HH20}*{Theorem 5.9}] \label{thm:MMP scaling fiber has gmm}
Let $(X,D)$ be a log canonical pair over $B$ and let $\pi\colon X\to Y$ be a projective morphism of normal quasi-projective varieties over $B$. Let $A\ge 0$ be an $\R$-divisor on $Y$ that is ample over $B$ such that $(X,D+\pi^*A)$ is log canonical. Assume that $(X,D)$ has a good minimal model over $Y$, and $K_X+D+\pi^*A$ is pseudo-effective over $B$, then $(X,D+\pi^*A)$ has a good minimal model over $B$.
\end{thm}

We also need the following statement, which is essentially a consequence of the negativity lemma (and the existence of MMP with scaling).

\begin{thm}[\cite{Bir12-lc-flip}*{Theorem 3.5}] \label{thm:birkar MMP contract exceptional}
Let $(X,D)$ be a $\bQ$-factorial dlt pair over $Z$ such that $X\to Z$ is birational. Assume that
\[K_X+D\simover{Z} M^+ - M^-\]
where $M^+, M^-\ge 0$ have no common components and $M^+$ is exceptional over $Z$. Then any $(K_X+D)$-MMP over $Z$ with scaling of a relatively ample divisor contracts $M^+$ after finitely many steps.
\end{thm}

We should remark that the theorem only implies that after finitely many steps of the $(K_X+D)$-MMP over $Z$ we get a model $X\dashrightarrow Y$ such that $M^+_Y=0$. However, $(Y,D_Y)$ is not necessarily a minimal model over $Z$ (this would require the termination of the $(K_X+D)$-MMP over $Z$, which is still open).

\subsection{Families of stable pairs} \label{sec:stable family}

In this subsection, we recall some definitions and results on families of stable pairs in \cite{Kol23}. In general, it is quite subtle to give the correct definition of families of stable pairs. Here we mainly focus on the case when the bases of the families are reduced. 

We start with the definition of relative Mumford divisors. We only consider the case when the morphism is flat.

\begin{defi}[\cite{Kol23}*{Definition 4.68}]\label{defi:relative Mumford divisors}
Let $f\colon X\to B$ be a flat morphism with reduced, connected and $S_2$ fibers of pure dimension $d$ over a reduced scheme $B$. A \emph{relative Mumford divisor} on $X$ is a divisor $D$ on $X$ such that
\begin{enumerate}
    \item[(1)] (Equidimensionality) Every irreducible component of $\Supp D$ dominates an irreducible component of $B$, and all nonempty fibers of $\Supp D\to B$ are of pure dimension $d-1$.

    \item[(2)] (Mumford condition) The morphism $f$ is smooth at the generic points of $X_b\cap\Supp D$ for every $b\in B$.

    \item[(3)] (Generic Cartier condition) The divisor $D$ is Cartier at the generic points of $X_b\cap\Supp D$ for every $b\in B$.
\end{enumerate}
\end{defi}

Note that when $B$ is normal, conditions (1) and (2) imply condition (3) by \cite{Kol23}*{Theorem 4.4}.

We will usually fix a marking when we deal with stable pairs. So next we recall the definition of families of marked pairs. Fix a coefficient vector $\va=(a_1,\ldots,a_n)$ such that $0\le a_i\le 1$ for every $1\le i\le n$.

\begin{defi}\label{defi:$n$-marked families}
A family of varieties marked with $n$ divisors, or an \emph{$n$-marked family} over a reduced scheme $B$ is $f\colon(X;D_1,\ldots,D_n)\to B$ such that $f$ is flat with demi-normal and connected fibers, and $D_i$ is an effective, relative Mumford divisor for every $1\le i\le n$.
\end{defi}

\begin{defi}[\cite{Kol23}*{Definition 8.4}]\label{defi:families of marked pairs}
A \emph{family of marked pairs with coefficient vector $\va=(a_1,\ldots,a_n)$} over a reduced scheme $B$ is a family of pairs $f\colon(X,\sum^{n}_{i=1}a_iD_i)\to B$ such that $f\colon(X;D_1,\ldots,D_n)\to B$ is an $n$-marked family.
\end{defi}

Next, we define locally stable families of pairs.

\begin{defi}[\cite{Kol23}*{Theorem-Definition 4.7}]\label{defi:locally stable families of pairs}
A projective family of marked pairs $f\colon(X,\Delta=\sum^{n}_{i=1}a_iD_i)\to B$ with coefficient vector $\va=(a_1,\ldots,a_n)$ over a reduced scheme $B$ is \emph{locally stable} if $K_{X/B}+\Delta$ is $\R$-Cartier, and $(X_b,\Delta_b)$ is slc for every $b\in B$. We also say that $f\colon(X,\Delta)\to B$ is a \emph{locally stable family of pairs with coefficient vector $\va=(a_1,\ldots,a_n)$}, or simply that $f$ is \emph{locally stable}.
\end{defi}

Note that our definition is slightly different from \cite{Kol23}*{Theorem-Definition 4.7}, since here we also have a marking (the definition in \emph{loc. cit.} should be thought of as the unmarked version). Since almost all the families of pairs we consider in this paper are families of marked pairs, we insist on the marking in all the definitions.

When the base is smooth, we have the following simple characterization of locally stable families.

\begin{thm}[\cite{Kol23}*{Theorem 4.54 and Corollary 4.55}] \label{thm:locally stable smooth base}
A projective family of pairs $f\colon(X,\Delta)\to B$ over a smooth scheme $B$ is locally stable if and only if for every simple normal crossing divisor $D$ on $B$, the pair $(X,\Delta+f^*D)$ is slc.
\end{thm}

Similarly, we can define stable families.

\begin{defi}\label{defi:stable families}
A family of marked pairs $f\colon(X,\Delta=\sum^{n}_{i=1}a_iD_i)\to B$ with coefficient vector $\va=(a_1,\ldots,a_n)$ over a reduced scheme $B$ is \emph{stable} if $f\colon(X,\Delta)\to B$ is locally stable, and $K_{X/B}+\Delta$ is ample over $B$. We also say that $f\colon(X,\Delta)\to B$ is a \emph{stable family of pairs with coefficient vector $\va=(a_1,\ldots,a_n)$}, or simply that $f$ is \emph{stable}.
\end{defi}

The foundational theorem in the moduli theory of higher-dimensional varieties is that stable pairs form a proper moduli.

\begin{thm}[\cite{Kol23}*{Theorems 4.1 and 4.8}] \label{thm:moduli of stable pairs}
Fix a vector $\va=(a_1,\ldots,a_n)$ such that $0< a_i\le 1$ for every $1\le i\le n$, a positive integer $d$, and a positive number $v$. Then there exists a proper Deligne--Mumford stack $\cK_{\va,d,v}$ which represents the moduli problem of stable families $f\colon(X,\sum^{n}_{i=1}a_iD_i)\to B$ with fibers of dimension $d$ and volume $v$ for reduced bases $B$. 
\end{thm}

For ease of notation we will omit $d$ in $\cK_{\va,d,v}$ as the dimension is often irrelevant in our discussion. We define $\cK_\va:=\cup_v\cK_{\va,v}$. By convention, we extend the definition of $\cK_\va$ to coefficient vectors $\va$ where some of the $a_i$ are allowed to be zero: in this case, we simply set $\cK_\va$ to be $\cK_{\va'}$ where $\va'$ is obtained by removing the zero entries in $\va$.

A direct consequence of the representability of the moduli problem is the following descent property.

\begin{lemma} \label{lem:descend stable family}
Let $\varphi\colon B'\to B$ be a proper morphism between normal varieties with connected fibers, and let $(X',D')\to B'$ be a stable family. Assume that for every $b\in B$, the induced family $(X',D')\times_{B'} \varphi^{-1}(b)$ is a trivial family. Then there exists a stable family $(X,D)\to B$ such that $(X',D')\cong (X,D)\times_B B'$.
\end{lemma}

Here we equip $\varphi^{-1}(b)$ with the reduced scheme structure. We say that a stable family $(X,D)\to B$ is trivial if $(X,D)\cong (X_b,D_b)\times B$ over $B$ for some (or any) $b\in B$.

\begin{proof}
This follows from Theorem \ref{thm:moduli of stable pairs} but can also be checked directly. By assumption and standard Hilbert scheme argument, the flat morphism $X'\to B'$ descends to $X\to B$. Since $D'$ has no variations along the fibers of $\varphi$, we see that each component of $D'$ maps onto a divisor in $X$. Let $D$ be the natural pushforward of $D'$. Since the base $B$ is normal, the relative Mumford divisor condition can be checked on the fibers; since it is true for $D'$, the same holds for $D$. Similarly, we know that $(X_b,D_b)$ is slc for any $b\in B$. By construction, we have $(X',D')\cong (X,D)\times_B B'$, and the relatively ample divisor $K_{X'/B'}+D'$ descends to $X$. Its descent is necessarily $K_{X/B}+D$, hence $K_{X/B}+D$ is $\R$-Cartier and ample over $B$. These imply that the family $(X,D)\to B$ is stable.
\end{proof}

\section{Existence of log canonical models}\label{sec:MMP results}

In this section, we prove several results on the existence of good minimal models and log canonical models for log canonical pairs. They form the first main ingredient in the construction of the reduction morphisms in Theorem \ref{main thm:weight reduction}.

\subsection{Existence of good minimal models}

First we show that abundance is preserved when reducing coefficients away from zero.

\begin{thm} \label{thm:gmm smaller coef}
Let $(X,D)$ be a log canonical pair over $B$ and let $0\le \Delta\le D$ be an $\R$-divisor. Assume that $(X,D)$ has a good minimal model over $B$, $K_X+\Delta$ is $\R$-Cartier, and $\Supp(\Delta)=\Supp(D)$. Then $(X,\Delta)$ has a good minimal model or a Mori fiber space over $B$.
\end{thm}

\begin{rem}
Here the condition that $\Supp(\Delta)=\Supp(D)$ is crucial. Without this assumption, the statement becomes equivalent to the full MMP conjecture (including abundance) that any log canonical pair has either a good minimal model or a Mori fiber space. This is because for any log canonical pair $(X,\Delta)$ we can always choose some sufficiently ample divisor $H\ge 0$ such that $(X,D:=\Delta+H)$ is a stable pair. In particular, it satisfies all the other conditions of the above theorem.
\end{rem}

Before we prove this theorem, we first consider a special case.

\begin{lemma} \label{lem:gmm smaller coef, K_X+D nef}
In the setting of Theorem \ref{thm:gmm smaller coef}, assume that $K_X+D$ is semiample over $B$. Then $(X,\Delta)$ has a good minimal model or a Mori fiber space over $B$.
\end{lemma}

\begin{proof}
If $K_X+\Delta$ is not pseudo-effective over $B$, then it has a Mori fiber space over $B$ by \cite{HH20}*{Theorem 1.7}. Thus we may assume that $K_X+\Delta$ is pseudo-effective over $B$. We need to show that it has a good minimal model over $B$.

Let $\pi\colon X\to Z$ be the ample model of $K_X+D$ over $B$. Then $K_X+D\sim_\R \pi^*A$ for some $\R$-divisor $A$ that is ample over $B$. Since $0\le \Delta\le D$ and $\Supp(\Delta)=\Supp(D)$ by assumption, there exists some $\R$-divisor $0\le \Gamma\le \Delta$ on $X$ such that $\Delta$ is a nontrivial convex combination of $\Gamma$ and $D$, i.e. 
\[\Delta=t\Gamma+(1-t)D\]
for some $t\in (0,1)$. We then have
\[K_X+\Delta = t(K_X+\Gamma)+(1-t)(K_X+D) \simover{B} t(K_X+\Gamma+\pi^*H)\]
where $0\le H\simover{B}\frac{1-t}{t}A$ such that $(X,\Gamma+\pi^*H)$ is log canonical. Therefore, it suffices to show that $(X,\Gamma+\pi^*H)$ has a good minimal model over $B$. As $K_X+\Gamma+(D-\Gamma)\simover{Z} 0$, by Theorem \ref{thm:birkar gmm} we see that $(X,\Gamma)$ has a good minimal model over $Z$. Since $H$ is ample over $B$, the statement we want follows from Theorem \ref{thm:MMP scaling fiber has gmm}.
\end{proof}

We now consider the general case of Theorem \ref{thm:gmm smaller coef}.

\begin{proof}[Proof of Theorem \ref{thm:gmm smaller coef}]
As in the proof of Lemma \ref{lem:gmm smaller coef, K_X+D nef}, we may assume that $K_X+\Delta$ is pseudo-effective over $B$ (otherwise it has a Mori fiber space over $B$). We first reduce to the case when the birational contraction to some weak log canonical model of $(X,D)$ over $B$ is a morphism. By assumption, $(X,D)$ has a good minimal model $(Z,D_Z)$ over $B$. Denote by $\pi\colon X\dashrightarrow Z$ the induced birational contraction over $B$. Choose a common log resolution
\[
\xymatrix{
& W \ar[ld]_\varphi \ar[rd]^\psi & \\
X \ar@{-->}^{\pi}[rr] & & Z\\
}
\]
and let $F$ be the sum of the $\varphi$-exceptional divisors. Since $(X,D)$ is log canonical, we have $K_W+D_W+F - \varphi^*(K_X+D)$ is effective and $\varphi$-exceptional. In particular, we see that $(Z,D_Z)$ is a weak log canonical model of $(W,D_W+F)$ over $B$. By Lemma \ref{lem:compare gmm existence on bir model}, we also know that $(X,\Delta)$ has a good minimal model over $B$ if and only if $(W,\Delta_W+F)$ has a good minimal model over $B$. Thus in order to prove the theorem we may replace $(X,D)$ by $(W,D_W+F)$ and $\Delta$ by $\Delta_W+F$. In particular, we may assume that the induced birational map $\pi\colon X\dashrightarrow Z$ to the weak log canonical model $(Z,D_Z)$ over $B$ is a morphism. Moreover, we may assume that $(X,D)$ is $\Q$-factorial dlt.

We write
\[K_X+D = \pi^*(K_Z+D_Z) + E \simover{Z} E\]
where $E$ is $\pi$-exceptional. As $(Z,D_Z)$ is a weak log canonical model of $(X,D)$ over $B$, we also know that $E\ge 0$. Recall that $0\le\Delta\le D$, thus from
\[K_X+\Delta = K_X+D - (D-\Delta) \simover{Z} E - (D-\Delta)\]
we get
\[K_X+\Delta \simover{Z} M^+ - M^-\]
where $M^+, M^-\ge 0$ have no common components, $M^-\le D-\Delta$ and $M^+\le E$. In particular, $M^+$ is exceptional over $Z$. Since $(X,D)$ is $\Q$-factorial dlt, $(X,\Delta)$ is $\Q$-factorial dlt as well. By Theorem \ref{thm:birkar MMP contract exceptional}, the $(K_X+\Delta)$-MMP over $Z$ with scaling of some ample divisor contracts $M^+$ after finitely many steps. Let $X\dashrightarrow Y$ be the corresponding birational contraction over $Z$, let $\varphi\colon Y\to Z$ be the induced morphism and let $\Delta'=M^-_Y\ge 0$. Then $M^+_Y=0$ and by pushing forward the previous identities to $Y$ we obtain
\[K_Y+\Delta_Y+\Delta' = \varphi^*(K_Z+D_Z).\]
This implies that $(Z,D_Z)$ is also a weak log canonical model of $(Y,D':=\Delta_Y+\Delta')$ over $B$, and moreover, $K_Y+D'$ is semiample over $B$. From $\Delta_Y\le D'\le D_Y$ and $\Supp(\Delta)=\Supp(D)$ we also see that $\Supp(\Delta_Y)=\Supp(D')$, hence Lemma \ref{lem:gmm smaller coef, K_X+D nef} applies to $(Y,D')$ and $\Delta_Y$, and we deduce that $(Y,\Delta_Y)$ has a good minimal model over $B$. Since $X\dashrightarrow Y$ is part of a $(K_X+\Delta)$-MMP over $Z$, it follows from Lemma \ref{lem:compare gmm existence on bir model} (and the remark thereafter) that $(X,\Delta)$ has a good minimal model over $B$ as well. The proof is now complete.
\end{proof}

\subsection{Log canonical models}

When we decrease coefficients from a locally stable family of pairs, we often end up with pairs $(X,\Delta)$ such that $K_X+\Delta$ is no longer $\R$-Cartier. Nonetheless, they are of log canonical type (sometimes also called potentially log canonical), i.e. there exists some $\R$-divisor $D\ge \Delta$ on $X$ such that $(X,D)$ is log canonical. In this subsection, we define and study the log canonical models of such pairs. A key input is the following construction.

\begin{lemma}[\cite{Kol18-log-plurigenera}*{Proposition 19}] \label{lem:small Q-Cartier modification}
Let $(X,D)$ be a log canonical pair and let $\Gamma$ be an effective $\R$-divisor supported on $D$. Then there exists a small birational projective morphism $\pi\colon Y\to X$ such that $-\Gamma_Y$ is $\R$-Cartier and $\pi$-ample.
\end{lemma}

\begin{proof}
We reproduce the proof for the reader's convenience. 
Let $W\to X$ be a dlt modification of $(X,D)$ and let $E$ be the sum of its exceptional divisors. Then $(W,D_W+E)$ is dlt and
\[K_W+D_W+E\simover{X} 0.\]
Choose some $0<\varepsilon\ll 1$ such that $D\ge \varepsilon \Gamma$. By Theorem \ref{thm:birkar gmm}, the pair $(W,D_W-\varepsilon \Gamma_W +E)$ has a good minimal model over $X$. Let $\pi\colon Y\to X$ be the corresponding log canonical model over $X$. Note that $K_Y+D_Y+E_Y\simover{X} 0$, hence
\[K_{Y}+D_{Y}-\varepsilon \Gamma_{Y}+E_{Y}\simover{X} -\varepsilon \Gamma_{Y}\]
is $\R$-Cartier and $\pi$-ample by construction. Let us show that $\pi$ is small; in other words, $E_Y=0$. To see this, we observe that as $-\Gamma_Y$ is $\pi$-ample, the $\pi$-exceptional locus is necessarily contained in the support of $\Gamma_Y$. Since $E_Y$ and $\Gamma_Y$ do not have any common component, we deduce that $E_Y=0$. This finishes the proof.
\end{proof}

Repeating this construction for each component of $D$, we get

\begin{coro} \label{cor:partial Q-factorial modification}
Let $(X,D)$ be a log canonical pair. Then there exists a small birational projective morphism $\pi\colon Y\to X$ such that $K_Y$ and every irreducible component of $D_Y$ are all $\Q$-Cartier.
\end{coro}

By \cite{BCHM10}, every variety that underlies a klt pair has a small $\Q$-factorial modification, but this is no longer true in the log canonical case. The construction above thus provides a small ``partial $\Q$-factorial'' modification for log canonical pairs. We should point out that, if the pair $(X,D)$ is only slc, then it may not be possible to find a small modification as in Corollary \ref{cor:partial Q-factorial modification}. In fact, if $X$ is a surface, then the only small birational modification is the identity, but not every irreducible component of $D$ is $\bQ$-Cartier (e.g. if the component intersects some double locus where two components of $X$ meet).

\medskip

We next define log canonical models and good minimal models for pairs of log canonical type.

\begin{defi}\label{defi:log canonical models}
Let $(X,\Delta)$ be a pair of log canonical type over $B$. Suppose that $X\dashrightarrow Z$ is a birational contraction over $B$. Then $(Z,\Delta_Z)$ is called a log canonical model (resp. good minimal model) of $(X,\Delta)$ over $B$ if there exists a small birational projective morphism $\pi\colon Y\to X$ such that $K_Y+\Delta_Y$ is $\R$-Cartier and $(Z,\Delta_Z)$ is a log canonical model (resp. good minimal model) of $(Y,\Delta_Y)$ over $B$.
\end{defi}

Similarly, we can define weak log canonical models for pairs of log canonical type. We leave the details to the reader.

\begin{rem}\label{rem:log canonical models are well-defined}
   The existence of such a small birational projective morphism is guaranteed by Lemma \ref{lem:small Q-Cartier modification}. Since $K_Y+\Delta_Y$ is $\R$-Cartier, $(Y,\Delta_Y)$ is a log canonical pair. Thus we can talk about the log canonical model (resp. good minimal models) of $(Y,\Delta_Y)$ over $B$. By Lemma \ref{lem:compare gmm existence on bir model}, the existence of the log canonical model (resp. good minimal models) is independent of the choice of the small birational projective morphisms, and the log canonical model of $(X,\Delta)$ over $B$, if exists, is unique. Moreover, it coincides with the standard definition of the log canonical model (resp. good minimal models) of $(X,\Delta)$ over $B$ if $K_X+\Delta$ is $\R$-Cartier. 
\end{rem}

\begin{rem}\label{rem:different ways to define log canonical models are the same}
Let $g\colon Y\to X$ be any birational projective morphism such that $(Y,\Delta_Y+E)$ is a log canonical pair where $E$ is the sum of all the $g$-exceptional divisors. Then by Lemma \ref{lem:compare gmm existence on bir model}, the log canonical model of $(X,\Delta)$ over $B$ defined above is the same as the log canonical model of $(Y,\Delta_Y+E)$ over $B$. Since $(X,\Delta)$ is a pair of log canonical type over $B$, there exists an $\R$-divisor $D\ge \Delta$ on $X$ such that $(X,D)$ is log canonical. In practice, we usually choose $g$ to be a dlt modification of $(X,D)$ or a log resolution of $(X,D)$. It is convenient to work with these special birational projective morphisms sometimes.
\end{rem}

The results from the previous subsection then have straightforward generalizations to the non-$\R$-Cartier setting. In particular, we have

\begin{thm} \label{thm:lc model exists non-R-Cartier}
Let $(X,D)$ be a pair of log canonical type over $B$ and let $0\le \Delta\le D$ be an $\R$-divisor. Assume that $(X,D)$ has a good minimal model over $B$, $K_X+\Delta$ is pseudo-effective (resp. big) over $B$, and $\Supp(\Delta)=\Supp(D)$. Then $(X,\Delta)$ has a good minimal model (resp. log canonical model) over $B$.
\end{thm}

Here we say that $K_X+\Delta$ is pseudo-effective over $B$ if there exists a small birational projective morphism $Y\to X$ such that $K_Y+\Delta_Y$ is $\R$-Cartier and pseudo-effective over $B$. This notion is independent of the choice of the small birational projective morphisms.

\begin{proof}
By Corollary \ref{cor:partial Q-factorial modification}, there exists a small birational projective morphism $\pi\colon Y\to X$ such that $K_Y+D_Y$ and $K_Y+\Delta_Y$ are both $\R$-Cartier. By Remark \ref{rem:log canonical models are well-defined}, we see that $(Y,D_Y)$ has a good minimal model over $B$. Thus our assumptions imply that $(Y,D_Y)$ and $\Delta_Y$ satisfy the assumptions of Theorem \ref{thm:gmm smaller coef}, which then implies that $(Y,\Delta_Y)$ has a good minimal model over $B$. If $K_X+\Delta$ is big over $B$, so is $K_Y+\Delta_Y$ (since $\pi$ is small), hence we also get a log canonical model of $(Y,\Delta_Y)$ over $B$. By definition, it follows that $(X,\Delta)$ has a good minimal model (resp. log canonical model if $K_X+\Delta$ is big over $B$) over $B$.
\end{proof}

\section{MMP for locally stable families} \label{sec:MMP locally stable}

\subsection{MMP with scaling}

The goal of this section is to develop an MMP theory for locally stable families. The main result is the following statement. This is another main ingredient in the construction of the reduction morphisms. 

\begin{thm} \label{thm:MMP locally stable family}
Let $f\colon (X,D)\to B$ be a locally stable family over a normal variety with normal generic fiber. Then for any $f$-ample $\R$-divisor $H$ on $X$, we may run the $(K_{X/B}+D)$-MMP over $B$ with scaling of $H$. Moreover, the MMP terminates with a good minimal model over $B$ if the generic fiber has a good minimal model.
\end{thm}

When $K_{X/B}+D$ is pseudo-effective over $B$ and the $(K_{X/B}+D)$-MMP over $B$ terminates, we get a minimal model $(X_m,D_m)\to B$ which is locally stable such that $K_{X_m/B}+D_m$ is nef over $B$. We say that it is a good minimal model of $(X,D)\to B$ if $K_{X_m/B}+D_m$ is semiample over $B$. If in addition $K_{X/B}+D$ is big over $B$, we may also define the log canonical model of $(X,D)\to B$ as the ample model of $K_{X_m/B}+D_m$ over $B$. It is necessarily also the ample model of $K_{X/B}+D$ over $B$. All these are analogous to Definition \ref{defi:good minimal model}.

To prove Theorem \ref{thm:MMP locally stable family}, we shall first take a resolution $\tB\to B$ of singularities, run the MMP after the base change, and descend every step of the MMP to $B$. As such, we need to analyze how maps such as extremal contractions and flips descend. The next several lemmas are designed to take care of these.

The first ingredient is the following invariance of plurigenera for good minimal models of locally stable families. By cohomology and base change, it eventually translates into the descent property of the extremal contractions.

\begin{lemma}\label{lem:h^0 constant smooth base}
Let $f\colon (X,D)\to B$ be a locally stable family over a connected reduced scheme with normal generic fibers. Assume that $D$ is a $\Q$-divisor, $K_{X/B}+D$ is nef over $B$ and every generic fiber of $f$ has a good minimal model. Then for any sufficiently divisible positive integer $m$, the dimension of $H^0(X_b, \cO_{X_b}(m(K_{X_b}+D_b)))$ is independent of $b\in B$.
\end{lemma}

\begin{proof}
We may freely shrink $B$ around any chosen point as needed. In particular, we may assume that $B$ is quasi-projective. Since the assumptions and the statement are preserved by any birational proper base change, we may also replace $B$ by a resolution of singularities and assume that $B$ is a smooth variety. Note that this implies that $X$ is normal, since the fibers of $X\to B$ are all reduced and $S_2$ while the generic fiber is normal, and together with the smoothness of the base these properties imply that $X$ is smooth in codimension one and $S_2$. Then $(X,D)$ is log canonical by Theorem \ref{thm:locally stable smooth base}, and every log canonical center of $(X,D)$ dominates $B$ by \cite{Kol23}*{Corollary 4.56}. By \cite{HX13}*{Theorem 1.1} and \cite{Has19}*{Theorem 1.2}, we deduce that $K_{X/B}+D$ is semiample over $B$. Let $b\in B$ and let $\Gamma=\sum\Gamma_i$ be any simple normal crossing divisor on $B$ containing $b$ as a $0$-dimensional stratum where $\Gamma_i$ are the irreducible components of $\Gamma$. By Theorem \ref{thm:locally stable smooth base}, the pair $(X,D+f^*\Gamma)$ is log canonical. 

Let $\pi\colon X\to Y$ be the ample model of $K_{X/B}+D$ over $B$ and let $g\colon Y\to B$ be the induced morphism. In particular, $K_X+D\sim_{\Q,Y}0$. By the canonical bundle formula \cite{Kol07-canonical-bundle-formula}*{Theorem 8.5.1}, we may write 
\[K_X+D\sim_{\Q}\pi^*(K_Y+\Delta+M)\] 
where $(Y,\Delta+M)$ is a generalized pair over $B$ and $(Y,\Delta+g^*\Gamma+M)$ is generalized log canonical. By \cite{LX23}*{Theorem 4.10} and \cite{Kol13}*{Theorem 9.26}, any union (with the reduced scheme structure) of the generalized log canonical centers of $(Y,\Delta+g^*\Gamma+M)$ is seminormal. By \cite{LX23}*{Lemma 3.17(3)}, any intersection of unions of generalized log canonical centers is a union of generalized log canonical centers. Putting all these together, we deduce that any scheme-theoretic intersection of unions of generalized log canonical centers is reduced by \cite{Kol13}*{Lemma 10.21} and thus seminormal. Since $(Y,\Delta+g^*\Gamma+M)$ is generalized log canonical, each $g^*\Gamma_i$ is reduced and is a union of generalized log canonical centers of $(Y,\Delta+g^*\Gamma+M)$. Since $Y_b$ is a connected component of the scheme-theoretic intersection of all $g^*\Gamma_i$, we deduce that $Y_b$ is reduced and seminormal. 

We claim that $g$ is flat. Morally this follows from a generalized log canonical version of Theorem \ref{thm:locally stable smooth base}. As such a version does not seem available in the literature, we will deduce flatness using some results in \cite{Kol23}. Since $Y_b$ is geometrically reduced for every closed point $b\in B$, so is $Y_b$ for every (possibly non-closed) $b\in B$ since the set $\{b\in B\,|\,Y_b\text{ is geometrically reduced over }k(b)\}$ is constructible in $B$. Since $f\colon X\to B$ is flat, every fiber of $f$ is of pure dimension $\dim X-\dim B$. Since the dimensions of the fibers of $X\to Y$ and $Y\to B$ are upper semicontinuous and $(\dim X_b-\dim Y_b)+\dim Y_b=\dim X_b$ is constant, we deduce that every fiber of $g\colon Y\to B$ is of pure dimension $\dim Y-\dim B$. By \cite{Kol23}*{2.71.2}, $g$ is universally open. Thus the morphism $g$ is pure dimensional by \cite{Kol23}*{Claim 2.71.1}. Since $g$ is pure dimensional and has geometrically reduced fibers, $g$ is flat by \cite{Kol23}*{Lemma 10.58}.

Let $L=K_Y+\Delta+M$. Then by Lemma \ref{lem:pushforward O_X = O_Y when Y seminormal}, we have $(\pi_b)_*\cO_{X_b} \cong \cO_{Y_b}$ and hence the projection formula gives
\begin{equation} \label{eq:H^0 equal}
    H^0(X_b, \cO_{X_b}(m(K_{X_b}+D_b)))\cong H^0(X_b, \cO_{X_b}(m\pi_b^*L_b)) \cong H^0(Y_b,\cO_{Y_b}(mL_b))
\end{equation}
for any sufficiently divisible integer $m$. Note that $L$ is ample over $B$ by construction, hence $R^ig_*\cO_Y(mL)=0$ for any $i>0$ and any sufficiently divisible positive integer $m$. By cohomology and base change, it follows that $h^0(Y_b,\cO_{Y_b}(mL_b))$ is independent of $b\in B$, hence the lemma follows from \eqref{eq:H^0 equal}.
\end{proof}

The following lemma is used in the above proof.

\begin{lemma}\label{lem:pushforward O_X = O_Y when Y seminormal}
Let $f\colon X\to Y$ be a proper surjective morphism of reduced schemes of finite type over an algebraically closed field of characteristic $0$. Assume that $f$ has connected fibers and $Y$ is seminormal. Then $f_*\cO_X\cong\cO_Y$.
\end{lemma}

\begin{proof}
This is well known. We take a Stein factorization $f=h\circ g$, where $g\colon X\to Z$ is proper, $h\colon Z\to Y$ is finite, and $g_*\cO_X\cong\cO_Z$. In particular, $g$ is surjective. Since $f$ has connected fibers and is surjective, we deduce that $h$ is both injective and surjective. Since $X$ is reduced and $g_*\cO_X\cong\cO_Z$, $Z$ is also reduced. Thus $h\colon Z\to Y$ is a partial seminormalization by \cite{Kol13}*{Definition 10.11} and the remark thereafter, since we are working over an algebraically closed field of characteristic $0$. Since $Y$ is seminormal, we deduce that $h$ is an isomorphism and thus $f_*\cO_X\cong\cO_Y$.
\end{proof}

A direct consequence of Lemma \ref{lem:h^0 constant smooth base} is the following statement. Among other things, it guarantees that the extremal contractions in the MMP descend. For a related result when the log canonical divisor is relatively semiample and big, see \cite{Kol23}*{Proposition 8.35}.

\begin{lemma} \label{lem:ample model commutes with base change}
Let $f\colon (X,D)\to B$ be a locally stable family over a reduced scheme with normal generic fibers. Let $\tB\to B$ be a proper surjective morphism from a reduced scheme and let $\tf\colon(\tX,\tD)\to \tB$ be the locally stable family obtained by base change via $\tB\to B$. Assume that $K_{\tX/\tB}+\tD$ is nef over $\tB$ and every normal generic fiber of $\tf$ has a good minimal model. Then $K_{X/B}+D$ is semiample over $B$, and the formation of its ample model commutes with base change.
\end{lemma}

\begin{proof}
We may assume that every irreducible component of $\tB$ dominates an irreducible component of $B$. In particular, every generic fiber of $\tf$ is normal and has a good minimal model. After replacing $\tB$ by a resolution of singularities, we may further assume that $\tB$ is smooth. By the same argument as in the proof of Lemma \ref{lem:h^0 constant smooth base}, we deduce that $K_{\tX/\tB}+\tD$ is semiample over $\tB$. By Lemma \ref{lem:semiample reduce to Q-coef}, we may write $D$ as a convex combination $D=\sum c_i D_i$ for some $c_i > 0$, $\sum_i c_i =1$ and some effective $\Q$-divisors $D_i$ such that for each $i$, $(X,D_i)\to B$ is locally stable, $K_{\tX/\tB}+\tD_i$ is semiample over $\tB$, and the ample model of $K_{\tX/\tB}+\tD_i$ over $\tB$ coincides with that of $K_{\tX/\tB}+\tD$. If the lemma holds for all the families $(X,D_i)\to B$, then it also holds for $(X,D)\to B$. Therefore, in order to prove the lemma we may assume that $D$ has rational coefficients.

In particular, we may choose a positive integer $r$ such that $r(K_{\tX/\tB}+\tD)$ is Cartier and base point free over $\tB$. It follows that $r(K_{X_b}+D_b)$ is Cartier and base point free for any $b\in B$, thus the semiampleness of $K_{X/B}+D$ over $B$ follows if we can show that the formation of $f_*\cO_X(m(K_{X/B}+D))$ commutes with base change for any sufficiently divisible positive integer $m$. Since $B$ is reduced, by cohomology and base change it suffices to show that $h^0(X_b, \cO_{X_b}(m(K_{X_b}+D_b)))$ is locally constant with respect to $b\in B$ (see e.g. \cite{Har77}*{Proof of Corollary 12.9}). Since the family $(X,D)\to B$ is locally stable, this condition can be checked after the base change via the proper surjective morphism $\tB\to B$, hence the result follows from Lemma \ref{lem:h^0 constant smooth base}.
\end{proof}

The following result is used in the above proof.

\begin{lemma} \label{lem:semiample reduce to Q-coef}
Let $f\colon (X,D)\to B$ be a locally stable family over a reduced scheme. Assume that $K_{X/B}+D$ is semiample over $B$. Then we can write $D=\sum_i c_i D_i$ as a convex combination for some $c_i > 0$, $\sum_i c_i =1$ and some effective $\Q$-divisors $D_i$ such that for each $i$, the family $(X,D_i)\to B$ is locally stable, $K_{X/B}+D_i$ is semiample over $B$, and the ample model of $K_{X/B}+D_i$ over $B$ coincides with that of $K_{X/B}+D$.
\end{lemma}

\begin{proof}
Let $g\colon X\to Y$ be the ample model of $K_{X/B}+D$ over $B$. By \cite{Kol23}*{Proposition 11.47}, there exist convex approximations $D=\sum c_i^n D_i^n$ ($n=1,2,\dots$) where $c_i^n > 0$, $\sum c_i^n = 1$, each $D_i^n$ is an effective $\Q$-divisor such that $(X,D_i^n)\to B$ is locally stable, and $\norm{D-D_i^n}\to 0$ as $n\to \infty$ (choose any norm on the vector space spanned by the components of $D$). Moreover, for each fixed $n$, one may also make the coefficients $c_i^n$ linearly independent over $\Q$. Since $K_{X/B}+D\sim_\R g^*H$ for some $\R$-Cartier $\R$-divisor $H$ that is ample over $B$, by \cite{Kol23}*{Claim 11.43.2} we first see that $K_{X/B}+D_i^n\sim_\Q g^*H_i^n$ for some $\Q$-Cartier $\Q$-divisor $H_i^n$ on $Y$; as $n\to \infty$, we then have $\norm{H-H_i^n}\to 0$, and thus $H_i^n$ is ample over $B$ for sufficiently large $n$. In particular $Y$ is also the ample model of $K_{X/B}+D_i^n$ over $B$, and the lemma holds for $c_i := c_i^n$ and $D_i := D_i^n$ by taking some sufficiently large $n$.
\end{proof}

We next need to show that flips descend. The main ingredient is the following lemma.

\begin{lemma} \label{lem:small Q-Cartier modification over normal base}
Let $f\colon (X,D)\to B$ be a locally stable family over a normal variety with normal generic fiber. Then for any effective $\R$-divisor $\Gamma$ supported on $D$, there exists a small birational projective morphism $\pi\colon Y\to X$ such that:
\begin{enumerate}
    \item $(Y,D_Y)\to B$ is locally stable with normal generic fiber, and
    \item $-\Gamma_Y$ is $\R$-Cartier and $\pi$-ample.
\end{enumerate}
\end{lemma}

\begin{rem}
    We remark that this is the only place in the MMP where we need the base $B$ to be normal; in fact, the examples in \cite{ABIP23}*{Section 8.1} can be used to show that flips do not descend if the base is only reduced (or even seminormal). On the other hand, we expect the lemma to hold when $X\to B$ is only quasi-projective, although we do not have a proof. 
\end{rem}

\begin{proof}
First note that the small birational modification $\pi$, if exists, is uniquely determined by the given conditions through a Proj construction, thus the question is of local nature and we may freely shrink $B$ as needed. Thus we may assume that $B$ is quasi-projective. Let $\varphi\colon B'\to B$ be a resolution of singularities and let $(X',D')\to B'$ denote the base change of the family $(X,D)\to B$. Let $\Gamma'$ be the (Weil-divisor) pullback of $\Gamma$ (in the sense of \cite{Kol23}*{Definition 4.2.7}).

Since $B'$ is smooth, the pair $(X',D')$ is log canonical by Theorem \ref{thm:locally stable smooth base}. Thus by Lemma \ref{lem:small Q-Cartier modification}, there exists a small birational projective morphism $\pi'\colon Y'\to X'$ such that $-\Gamma'_{Y'}$ is $\R$-Cartier and $\pi'$-ample. Since
\begin{equation} \label{eq:pullback by a small map}
    K_{Y'}+D'_{Y'} = \pi'^*(K_{X'}+D'),
\end{equation}
the family $(Y',D'_{Y'})\to B'$ is also locally stable by another application of Theorem \ref{thm:locally stable smooth base}. The plan is to descend the family $(Y',D'_{Y'})\to B'$ to the original base $B$.

Choose some $0<\varepsilon\ll 1$ and some general relative hyperplane section $H$ on $X$ that is sufficiently ample over $B$. Let $\Delta=D-\varepsilon \Gamma\ge 0$ and let $H'$ (resp. $\Delta'$) be the pullback of $H$ (resp. $\Delta$) to $X'$. Then it is not hard to see that
\[K_{Y'/B'}+\Delta'_{Y'}+H'_{Y'} = \pi'^*(K_{X'/B'}+D'+H') - \varepsilon \Gamma'_{Y'}\]
is ample over $B'$ and by Bertini theorem we may also assume that $(Y',\Delta'_{Y'}+H'_{Y'})\to B'$ is a stable family (i.e. the fibers are slc). If we can show that for any closed point $b\in B$, the induced stable family over $\varphi^{-1}(b)\subseteq B'$ is a trivial family (as before we treat $\varphi^{-1}(b)$ as a reduced scheme), then by Lemma \ref{lem:descend stable family} the family $(Y',\Delta'_{Y'}+H'_{Y'})\to B'$ descends to a stable family $(Y,\Delta_Y+H_Y)\to B$ making the diagram commutative with Cartesian squares:
\[
\xymatrix{
(Y',\Delta'_{Y'}+H'_{Y'})\ar[r] \ar[d]_{\pi'} & (Y,\Delta_Y+H_Y) \ar[d]^{\pi} \\
(X',D'+H') \ar[r] \ar[d]_{f'} & (X,D+H) \ar[d]^{f} \\
B' \ar[r]^{\varphi} & B.
}
\]
The induced map $\pi\colon Y\to X$ is then the small birational modification we sought after. Note that $(Y,D_Y)\to B$ is automatically locally stable once we know its existence. Indeed, the morphism $Y\to B$ is flat by construction, $K_{Y/B}+D_Y$ is $\R$-Cartier as it is the pullback of $K_{X/B}+D$, and over a normal base the relative Mumford divisor condition can be checked on the fibers (see the remark after Definition \ref{defi:relative Mumford divisors}). Moreover, the fibers of $(Y,D_Y)\to B$ are also fibers of $(Y',D'_{Y'})\to B'$, hence they are slc. Hence $(Y,D_Y)\to B$ is locally stable.

Let us now show that $(Y',\Delta'_{Y'}+H'_{Y'})\times_{B'} \varphi^{-1}(b)$ is a trivial family over $\varphi^{-1}(b)$. In fact, it suffices to show that $(Y',D'_{Y'})\times_{B'} \varphi^{-1}(b)$ is a trivial family, since then $\Delta'_{Y'}$ (supported on $D'_{Y'}$) and $H'_{Y'}$ (itself a pullback from $X$) cannot have any variation over $\varphi^{-1}(b)$. 

To simplify the notation we denote $W:=\varphi^{-1}(b)$ and $(Z,G):=(Y',D'_{Y'})\times_{B'} \varphi^{-1}(b)$. The restriction of the morphism $\pi'\colon Y'\to X'$ to $Z\subseteq Y'$ induces a morphism
\[\psi\colon Z\to X'\times_{B'}\varphi^{-1}(b)\cong X_b\times W.\]
Since $-\Gamma'_{Y'}$ is $\pi'$-ample, we deduce that
\[\Exc(\pi')\subseteq \Supp(\Gamma'_{Y'})\subseteq \Supp(D'_{Y'})\]
and hence $\Exc(\psi)\subseteq \Supp(G)$. By restricting \eqref{eq:pullback by a small map} to $Z$ we also see that $(Z,G)$ is the crepant pullback of $(X_b,D_b)\times W$. Note that $\psi$ has connected fibers since $\pi'$ does; moreover, $W$ is connected since $B$ is normal. We may now conclude that the family $(Z,G)\to W$ is trivial by the following Lemma \ref{lem:family trivial}.
\end{proof}

\begin{lemma} \label{lem:family trivial}
Consider a birational morphism with connected fibers between locally stable families over a connected reduced scheme:
\[
\xymatrix{
(Z,G)\ar[rr]^{\psi} \ar[rd] & & (X,D) \ar[ld] \\
 & B &
}
\]
Assume that:
\begin{enumerate}
    \item $\Exc(\psi)\subseteq \Supp(G)$, 
    \item $\psi$ is crepant, i.e. $K_{Z/B}+G = \psi^*(K_{X/B}+D)$,
    \item $-G_0$ is $\psi$-ample for some effective $\R$-Cartier $\R$-divisor $G_0$ supported on $G$, and
    \item $(X,D)\to B$ is a trivial family.
\end{enumerate}
Then $(Z,G)\to B$ is a trivial family.
\end{lemma}

\begin{proof} 
The lemma is then proved just as in \cite{ABIP23}*{Proposition 6.4 and Lemma 6.6}. The main points are:
\begin{itemize}
    \item Koll\'ar's gluing theory \cite{Kol13}*{Proposition 5.3} ensures that $(Z,G)$ is determined by its normalization together with the gluing relation induced by the involution on the conductor. If the normalization is a trivial family, then as the exceptional locus of $\psi$ does not contain any generic point of the non-normal locus of any fiber (by assumption (1)), the gluing relation is just induced by that of $(X,D)$, hence is also constant in the family and it follows that $(Z,G)\to B$ is also necessarily a trivial family. So we may reduce to the case when $Z$ and $X$ are normal.
    \item Assumption (3) implies that $Z$ is determined, through a Proj construction, by the exceptional divisors extracted by $\psi$ and the coefficients of $G_0$.
    \item Assumptions (1) and (2) imply that the $\psi$-exceptional divisors have negative discrepancies with respect to $(X,D)$ (these can be defined e.g. over the generic points of $B$). The only way to get such divisors is by successively blowing up strata of a log resolution of $(X,D)$ (\cite{ABIP23}*{Lemma 6.6}). If $(X,D)\to B$ is a trivial family and the base is connected, then such $\psi$-exceptional divisors have no variation in the family. Thus the family $(Z,G)\to B$ is also trivial.
\end{itemize}
\end{proof}

The last ingredient we need is the fact that the family remains locally stable after each MMP step. We already see this for flips (Lemma \ref{lem:small Q-Cartier modification over normal base}), so it remains to check the case of extremal contractions. We shall prove a slightly more general statement, since it is also useful in Section \ref{sec:wall crossing} when we study the wall-crossing morphisms.

\begin{lemma} \label{lem:ample model stable}
Let $f\colon (X,D)\to B$ be a locally stable family over a reduced scheme with normal generic fibers. Assume that $K_{X/B}+D$ is semiample and big over $B$ and let $\pi\colon X\to Y$ be its ample model over $B$. Then $g\colon(Y,D_Y:=\pi_*D)\to B$ is stable.
\end{lemma}

When the base $B$ is seminormal, this is essentially proved in \cite{ABIP23}*{Theorem 5.1}. The stronger result above over a reduced base is one of the ingredients which allows us to construct the wall-crossing morphisms without passing to the seminormalization.

\begin{proof}
We need to verify the conditions in Section \ref{sec:stable family}. Since the question is local on $B$ we may freely shrink $B$ as needed. Thus we may assume that $B$ is connected.

The flatness of $Y\to B$ and the equidimensionality of $\Supp(D_Y)\to B$ are checked just as in \cite{ABIP23}*{Theorem 5.1}. Here we briefly recall the argument for the reader's convenience. Since the formation of the ample model commutes with base change by Lemma \ref{lem:ample model commutes with base change}, the fiber $Y_b$ of $g$ is the ample model of $K_{X_b}+D_b$, which is necessarily reduced and connected. By the same argument as in the proof of Lemma \ref{lem:h^0 constant smooth base}, $g$ has geometrically reduced fibers of pure dimension $\dim Y-\dim B$. Since $f$ is flat, $f$ is universally open. Thus $g$ is universally open since $\pi$ is surjective. Thus the morphism $g$ is pure dimensional by \cite{Kol23}*{Claim 2.71.1}. Since $g$ is pure dimensional and has geometrically reduced fibers, $g$ is flat by \cite{Kol23}*{Lemma 10.58}. Since $K_{X/B}+D$ is big over $B$, the morphism $\pi$ is birational and thus $\dim Y_b = d$ for all $b\in B$, where $d$ is the relative dimension of $f\colon X\to B$. If $\Delta$ is an irreducible component of $D_Y$, then it is the $\pi$-image of some irreducible component $\Gamma$ of $D$. Since $\Gamma$ dominates an irreducible component of $B$, so does $\Delta$. Moreover, $\dim \Delta_b \le \dim \Gamma_b = d-1$ with equality for general $b\in B$. Combined with the upper semicontinuity of fiber dimensions, we see that $\dim \Delta_b = \dim \Gamma_b = d-1$ for all $b\in B$. This gives the equidimensionality condition. 

We proceed to check the Mumford and the generic Cartier conditions (Definition \ref{defi:relative Mumford divisors}). By \cite{Kol23}*{Proposition 4.26}, these can be checked on a general relative hyperplane section on $Y$. By Bertini theorem and \cite{Kol23}*{Proposition 10.56}, we also know that the restriction of $(X,D)$ to the general member of a relative base point free linear system (e.g. the pullback of a general relative hyperplane section on $Y$) remains locally stable over $B$. Thus we may assume that $X\to B$ has relative dimension $1$. In particular, each fiber $X_b$ is a nodal curve and $\Supp(D_b)$ is contained in its smooth locus. By \cite{Has03}*{Theorem 3.6}, the ample model of $K_{X/B}+D$ over $B$ is a stable family. In particular, it satisfies the Mumford and the generic Cartier conditions.

To this end, we conclude that $(Y,D_Y)\to B$ is a family of marked pairs (in the sense of Definition \ref{defi:families of marked pairs}). It remains to check that it is a stable family (Definition \ref{defi:stable families}). By construction, we have $K_{X/B}+D = \pi^*(K_{Y/B}+D_Y)$ and $K_{Y/B}+D_Y$ is $\R$-Cartier. Since the formation of the ample model commutes with base change by Lemma \ref{lem:ample model commutes with base change}, the remaining conditions in Definition \ref{defi:stable families} can be checked after a base change. In particular, we may replace $B$ by a resolution of singularities and assume it is smooth. The statement then follows from \cite{Kol23}*{Corollary 4.57}.
\end{proof}

We are now in a position to run the MMP for locally stable families over a normal base.

\begin{proof}[Proof of Theorem \ref{thm:MMP locally stable family}]
Let $\tB\to B$ be a resolution of singularities and let $(\tX,\tD)$ be the locally stable family obtained from $(X,D)$ by base change to $\tB$. Our assumption guarantees that $\tX$ is normal, since the fibers of $\tX\to \tB$ are all reduced and $S_2$ while the generic fiber is normal, and these properties imply that $\tX$ is smooth in codimension one and $S_2$.

By Theorem \ref{thm:locally stable smooth base}, the pair $(\tX,\tD)$ is log canonical, hence we may run the $(K_{\tX/\tB}+\tD)$-MMP (equivalently, the $(K_{\tX}+\tD)$-MMP) over $\tB$ with scaling of $\tH$ (more details on this below), where $\tH$ is the pullback of $H$ to $\tX$. To prove the theorem, our goal is to show that every step in this MMP descends to $B$.

The steps in this (non-$\Q$-factorial) $(K_{\tX/\tB}+\tD)$-MMP with scaling over $\tB$ take the following form (see e.g. \cite{Kol21-MMP-non-Q-factorial}*{Definition 1}). Let $(\tX_i,\tD_i)$ be the birational model obtained at the $i$-th step (when $i=0$ we set $(\tX_0,\tD_0):=(\tX,\tD)$) and let $\tH_i$ be the pushforward of $\tH$ to $\tX_i$ which remains $\R$-Cartier as will be seen from the construction. Let $t_i\ge 0$ be the smallest nonnegative number such that $K_{\tX_i}+\tD_i+t_i \tH_i$ is nef over $\tB$. If $t_i=0$, the MMP terminates with a minimal model over $\tB$. Otherwise $t_i>0$ and we let $\tZ_i$ be the ample model of $K_{\tX_i}+\tD_i+t_i \tH_i$ over $\tB$, whose existence is guaranteed by \cite{Fuj11}*{Theorem 1.1(4)} (alternatively, it follows from Theorem \ref{thm:MMP scaling fiber has gmm}, since $(\tX_i,\tD_i+t_i\tH_i)$ is a minimal model of $(\tX,\tD+t_i \tH)$ over $\tB$). If $\dim \tZ_i < \dim \tX_i$, the MMP terminates with a Fano fibration over $\tB$. Otherwise, the morphism $\tvarphi_i \colon \tX_i\to \tZ_i$ is birational and the next step of the MMP factorizes as
\[
    \xymatrix{
        (\tX_i,\tD_i) \ar[r]^-{\tvarphi_i} \ar[rd]_{\tf_i} & \tZ_i \ar[d]  & (\tX_{i+1},\tD_{i+1}) \ar[l]_-{\tpsi_i} \ar[ld]^{\tf_{i+1}} \\
        & \tB &
    }
\]
where $\tpsi_i$ is a small birational projective morphism such that $-\tH_{i+1}$ is $\tpsi_i$-ample. Its existence is guaranteed by Lemma \ref{lem:small Q-Cartier modification}, since $(\tZ_i,\tvarphi_{i*}(\tD_i+t_i \tH_i))$ is log canonical for a general choice of $H\ge 0$ in its $\R$-linear equivalence class over $B$. We shall fix this choice of $H$ in the rest of the proof. In particular, since 
\begin{align*}
    K_{\tX_{i+1}}+\tD_{i+1}+(t_i-\varepsilon)\tH_{i+1} & = \tpsi_i^*\tvarphi_{i*}(K_{\tX_i}+\tD_i+t_i \tH_i) - \varepsilon \tH_{i+1} \\
    & = \tpsi_i^*(\mbox{ample}/\tB)-\varepsilon \tH_{i+1}
\end{align*}
we know that $K_{\tX_{i+1}}+\tD_{i+1}+(t_i-\varepsilon)\tH_{i+1}$ is $\tf_{i+1}$-ample for $0<\varepsilon\ll 1$ and hence $t_{i+1}<t_i$. By induction on $i$, this also implies that $K_{\tX_i}+\tD_i+t\tH_i$ is $\tf_{i}$-ample for $t_i<t<t_{i-1}$.

We need to descend this sequence of $(K_{\tX/\tB}+\tD)$-MMP with scaling over $\tB$ to $B$. By induction on $i$, we may assume that $(\tX_i,\tD_i+t_i\tH_i)\to \tB$ descends to a locally stable family $(X_i,D_i+t_iH_i)\to B$. By Lemma \ref{lem:ample model commutes with base change}, the ample model $\tvarphi_i\colon \tX_i\to \tZ_i$ descends to give the ample model $\varphi_i\colon X_i\to Z_i$ of $K_{X_i/B}+D_i+t_i H_i$ over $B$. Moreover, when $\varphi_i$ is birational, the induced family $(Z_i,\varphi_{i*}(D_i+t_i H_i))\to B$ is stable by Lemma \ref{lem:ample model stable}. By Lemma \ref{lem:small Q-Cartier modification over normal base}, there exists a small birational projective morphism $\psi_{i+1}\colon X_{i+1}\to Z_i$, which is necessarily the descent of $\tpsi_{i+1}\colon \tX_{i+1}\to \tZ_i$, such that $-H_{i+1}$ is $\psi_{i+1}$-ample and $(X_{i+1},D_{i+1}+t_i H_{i+1})\to B$ is a locally stable family, where $D_{i+1}$ (resp. $H_{i+1}$) is the pushforward of $D_i$ (resp. $H_i$). This finishes the construction of the MMP with scaling over $B$.

Finally, if the generic fiber of $(\tX,\tD)\to \tB$, which is also the generic fiber of $(X,D)\to B$, has a good minimal model, then since every log canonical center of $(\tX,\tD)$ dominates $\tB$ by \cite{Kol23}*{Corollary 4.56}, using \cite{HX13}*{Theorem 1.1} and \cite{Has19}*{Theorem 1.2} we conclude that $(\tX,\tD)$ has a good minimal model $(\tX_m,\tD_m)$ over $\tB$. From the previous discussion, we know that it descends to a minimal model $(X_m,D_m)$ of $(X,D)$ over $B$. By Lemma \ref{lem:ample model commutes with base change}, we see that $(X_m,D_m)$ is also a good minimal model over $B$. The proof is now complete.
\end{proof}

\subsection{Semi-log-canonical models}

If the locally stable family in Theorem \ref{thm:MMP locally stable family} does not have a normal generic fiber, then the MMP may not exist, since MMP does not apply to slc pairs in general (see e.g. \cites{Kol11-slc-no-MMP,AK19} for some related examples). However, if we assume that the generic fiber has a semi-log-canonical model, then the family still has a relative semi-log-canonical model. This is enough for most applications in our work.

\begin{thm} \label{thm:slc model exist}
Let $f\colon (X,D)\to B$ be a locally stable family over a normal variety and let $0\le \Delta\le D$ be an $\R$-divisor. Assume that the generic fiber of $(X,\Delta)\to B$ has a semi-log-canonical model. Then the family $(X,\Delta)\to B$ has a semi-log-canonical model.
\end{thm}

To make this precise, we need to define semi-log-canonical models for the families that appear in the above statement. If the generic fiber is normal and $K_{X/B}+\Delta$ is $\R$-Cartier, then $(X,\Delta)\to B$ is locally stable and its (semi-)log canonical model is just the log canonical model defined right after Theorem \ref{thm:MMP locally stable family}. The $\R$-Cartier assumption can be dropped from the definition in a way similar to Definition \ref{defi:log canonical models}.

\begin{defi} \label{defi:lc model for families of locally stable type}
Let $(X,D)\to B$ be a locally stable family over a normal variety with normal generic fiber and let $0\le \Delta\le D$ be an $\R$-divisor such that $K_{X/B}+\Delta$ is big on the generic fiber. Suppose that $X\dashrightarrow Z$ is a birational contraction over $B$ and $(Z,\Delta_Z)\to B$ is a stable family. Then $(Z,\Delta_Z)\to B$ is called a log canonical model of $(X,\Delta)\to B$ if there exists a small birational projective morphism $Y\to X$ such that $(Y,\Delta_Y)\to B$ is locally stable with normal generic fiber and $(Z,\Delta_Z)\to B$ is the log canonical model of $(Y,\Delta_Y)\to B$.
\end{defi}

\begin{rem} \label{rem:equivalent defn of lc model for family}
Such a small birational projective morphism $Y\to X$ always exists by Lemma \ref{lem:small Q-Cartier modification over normal base}. Since a stable family is determined by its generic fiber (thanks to the separatedness of the moduli of stable pairs), by Remark \ref{rem:log canonical models are well-defined} we see that the log canonical model $(Z,\Delta_Z)\to B$ is unique if it exists. Moreover, by applying Theorem \ref{thm:MMP locally stable family} to the family $(Y,\Delta_Y)\to B$ we deduce that $(X,\Delta)\to B$ has a log canonical model if and only if its generic fiber has one. Thus we can also define the log canonical model as a stable family over $B$ whose generic fiber is the log canonical model of the generic fiber of $(X,\Delta)\to B$. This interpretation will be convenient in some of our subsequent discussions.
\end{rem}

To define semi-log-canonical models in the general setting, the following observation is necessary.  

\begin{lemma} \label{lem:normalization locally stable}
Let $f\colon (X,D)\to B$ be a locally stable family over a normal variety. Then its normalization $(X^\nu,\Gamma+D^\nu)\to B$ (where $\Gamma$ is the conductor) is a disjoint union of locally stable families over $B$.
\end{lemma}

The main subtlety here is the flatness of $X^\nu\to B$.

\begin{proof}
By adding a general relative hyperplane section to $D$ we may assume that the family is stable. When $B$ is smooth, the lemma is a direct consequence of Theorem \ref{thm:locally stable smooth base}. In general, let $\varphi\colon \tB\to B$ be a resolution of singularities and let $(\tX,\tD)\to \tB$ be the stable family obtained by base change. Then its normalization $(\tX^\nu,\widetilde{\Gamma}+\tD^\nu)\to \tB$ is stable (more precisely, it is a disjoint union of stable families over $\tB$). We claim that it descends to a stable family over $B$. By Lemma \ref{lem:descend stable family}, it suffices to check that the family is trivial along the exceptional fibers of $\varphi$. 

Let $b\in B$, $W:=\varphi^{-1}(b)$ and $Z:=\tX^\nu\times_{\tB} W$. Note that $W$ is connected as $B$ is normal. As in the proof of Lemma \ref{lem:small Q-Cartier modification over normal base}, in order to prove the above claim we only need to show that $Z\to W$ is a trivial family, since the boundary part is supported on the preimage of $\Supp(D_b)\cup \Sing(X_b)$, thus cannot have any variation. Since $(\tX^\nu,\widetilde{\Gamma}+\tD^\nu)\to \tB$ is stable, the conductor $\widetilde{\Gamma}$ does not contain any irreducible component of any fiber of $\tX^\nu\to \tB$. As the normalization morphism $\tX^\nu\to \tX$ is an isomorphism away from $\widetilde{\Gamma}$, we deduce that it restricts to a partial normalization $Z_w\to X_b$ for any $w\in W$. The latter is uniquely determined if we know which codimension one nodes in $X_b$ are normalized, because $Z_w$ is $S_2$. Since there are only finitely many codimension one nodes and $W$ is connected, we see that $Z_w\to X_b$ is independent of $w\in W$. Thus the conditions of Lemma \ref{lem:descend stable family} are verified and we conclude that $(\tX^\nu,\widetilde{\Gamma}+\tD^\nu)\to \tB$ descends to a stable family $(X^\nu,\Gamma+D^\nu)\to B$, which is at least a partial normalization of $(X,D)\to B$.

It remains to check that $X^\nu$ is normal. Since the fibers of  $(X^\nu,\Gamma+D^\nu)\to B$ are slc, we see that $X^\nu\to B$ has $S_2$ fibers. Combined with the normality of $B$, we deduce from \cite{EGA-IV-2}*{Corollaire 6.4.2} that $X^\nu$ is $S_2$. Moreover, since $X^\nu\to B$ has reduced fibers of constant dimension as well as a normal generic fiber, we see that the singular locus of $X^\nu$ has codimension at least two. These together imply that $X^\nu$ is normal. Thus $(X^\nu,\Gamma+D^\nu)\to B$ is the normalization of $(X,D)\to B$ and it is stable by construction.
\end{proof}

By Koll\'ar's gluing theory \cite{Kol13}*{Chapter 5}, a stable family is determined by its normalization together with the induced involution on $(\Gamma^\nu,\Diff_{\Gamma^\nu}(D^\nu))$, where $\Gamma^\nu$ is the normalization of $\Gamma$ and $\Diff_{\Gamma^\nu}(D^\nu)$ denotes the different (see e.g. \cite{Kol13}*{Definition 4.2}).

\begin{defi} \label{defi:slc model}
Let $(X,D)\to B$ be a locally stable family over a normal variety and let $0\le \Delta\le D$ be an $\R$-divisor such that $K_{X/B}+\Delta$ is big on every irreducible component of the generic fiber. A stable family $(Y,\Delta_Y)\to B$ is called a semi-log-canonical (slc) model of $(X,\Delta)\to B$ if there exists a birational map $\varphi\colon X\dashrightarrow Y$ over $B$ such that the normalization of $(Y,\Delta_Y)$ is the log canonical model over $B$ of the normalization of $(X,\Delta)$.
\end{defi}

Since $X$ and $Y$ have multiple irreducible components here, birational maps between them are defined as in \cite{Kol13}*{Definition 1.11}. In particular, every irreducible component of $X$ dominates an irreducible component of $Y$, and the map $\varphi$ is an isomorphism over the generic points of $Y$. The definition also implies that the conductor on $Y$ is the strict transform of the conductor on $X$.

\begin{proof}[Proof of Theorem \ref{thm:slc model exist}]
Let $(X^\nu,\Gamma+D^\nu)\to B$ be the normalization of $(X,D)\to B$ as before and let $\tau$ be the induced involution on $(\Gamma^\nu,\Diff_{\Gamma^\nu}(D^\nu))$ where $\Gamma^\nu$ is the normalization of $\Gamma$. Let $\Delta^\nu$ be the strict transform of $\Delta$. By assumption, the generic fiber of the family $(X^\nu,\Gamma+\Delta^\nu)\to B$ has a log canonical model, hence the family $(X^\nu,\Gamma+\Delta^\nu)\to B$ also has a log canonical model $(Y^\nu,\Gamma_Y+\Delta^\nu_Y)\to B$ by Remark \ref{rem:equivalent defn of lc model for family}. To construct the slc model, we may now assume that $B$ is smooth. Indeed, if the slc model exists after base change via a resolution of singularities $\tB\to B$, then as a stable family, the slc model we obtain over $\tB$ is automatically a trivial family along the exceptional fibers of $\tB\to B$ (this is because these fibers are connected, and there are only finitely many stable pairs with a given partial normalization; c.f. the proof of Lemma \ref{lem:family trivial}), hence by Lemma \ref{lem:descend stable family} it descends to give the slc model over $B$. 

Thus we shall assume that $B$ is smooth. To proceed, we need to show that the induced birational involution
\[\tau_Y\colon (\Gamma_Y^\nu,\Diff_{\Gamma_Y^\nu}(\Delta^\nu_Y))\dashrightarrow (\Gamma_Y^\nu,\Diff_{\Gamma_Y^\nu}(\Delta^\nu_Y))\]
on the normalization $\Gamma_Y^\nu$ of $\Gamma_Y$ is a regular involution. This holds over the generic point of $B$, since the generic fiber of $(X,\Delta)\to B$ has an slc model by assumption. Note that $(Y^\nu,\Gamma_Y+\Delta^\nu_Y)\to B$ is a stable family by construction and hence $(\Gamma_Y^\nu,\Diff_{\Gamma_Y^\nu}(\Delta^\nu_Y))\to B$ is a stable family by Theorem \ref{thm:locally stable smooth base} and adjunction. By the same argument as in the proof of \cite{HX13}*{Lemma 7.2}, we then deduce that $\tau_Y$ is regular.

We conclude by using Koll\'ar's gluing theory. Note that $(X^\nu,\Gamma+\Delta^\nu)$ is log canonical by Theorem \ref{thm:locally stable smooth base}, hence $(Y^\nu,\Gamma_Y+\Delta^\nu_Y)$ is also log canonical. By \cite{Kol13}*{Corollary 5.33 and Theorem 5.38}, there exists a stable family $(Y,\Delta_Y)\to B$ whose normalization is $(Y^\nu,\Gamma_Y+\Delta^\nu_Y)\to B$. This completes the proof.
\end{proof}

From the above proof, we see that the slc model is unique if it exists. By the same argument, we also get the following corollary.

\begin{coro} \label{cor:extend family}
Let $B$ be an open subset of a normal variety $\oB$, and let $(X,D)\to B$ be a stable family. Assume that its normalization $(X^\nu,\Gamma+D^\nu)\to B$ extends to a stable family $(\oX^\nu,\oGamma+\oD^\nu)\to \oB$. Then $(X,D)\to B$ extends uniquely to a stable family $(\oX, \oD)\to \oB$ whose normalization is $(\oX^\nu,\oGamma+\oD^\nu)\to \oB$.
\end{coro}

\section{Finiteness of ample models}\label{sec:Finiteness of ample models}

In \cite{BCHM10}*{Corollary 1.1.5}, results on the finiteness of ample models are presented. To study the wall crossing for moduli of stable pairs in the log canonical setting, we need to establish the finiteness of ample models in a more general setting. This is the main task of this section.

Let us fix some notation. Let $f\colon X\to B$ be a projective morphism of normal quasi-projective varieties. Let $D_1,\ldots,D_n$ be prime divisors on $X$. Consider the vector space
$$V:=\bigoplus^n_{i=1}\R D_i \cong \R^n.$$
For any $D=\sum^n_{i=1}a_iD_i$, we define
$$\norm{D}:=\max_{i}\{|a_i|\}$$
which gives a norm on $V$. For a log canonical pair $(X, D)$ over $B$, we define its ample model over $B$ as the ample model of $K_X+D$ over $B$. 

We follow \cite{BCHM10}*{Section 3.7} for definitions on convex geometry. A polytope (resp. rational polytope) in $\R^n$ is the convex hull of a finite set of points (resp. rational points). A finite (resp. rational) polyhedral decomposition of a (resp. rational) polytope $\cP$ is a finite decomposition $\cP=\cup_{i=1}^p \cP_i$ such that each $\cP_i$ is a (resp. rational) polytope, every face of $\cP_i$ coincides with $\cP_j$ for some $j$, and for all $1\le i,j\le p$, the intersection of $\cP_i$ and $\cP_j$ is either empty or a common face of both $\cP_i$ and $\cP_j$. Usually the interior $\cP_i^\circ$ of the polytopes $\cP_i$ are called the chambers, while their faces are called the walls. The main result of this section is the following.

\begin{thm} \label{thm:chamber decomp by ample model}
Notation as above. Let $\cP\subseteq V$ be a rational polytope such that $(X, D)$ is a log canonical pair and has a good minimal model over $B$ for every $D\in\cP$. 

Then there exist a finite rational polyhedral decomposition $\cP=\cup_{i=1}^p \cP_i$, together with finitely many rational maps $\psi_j\colon X\dashrightarrow Z_j$ over $B$, $1\le j\le q$, such that: for any $1\le i\le p$, there exists some $1\le j\le q$ such that $\psi_j\colon X\dashrightarrow Z_j$ is the ample model of $(X, D)$ over $B$ for any $D\in \cP_i^\circ$.
\end{thm}

The proof largely follows the general inductive strategy of \cite{BCHM10}. The key step is Proposition \ref{prop:local cone structure}, which essentially implies that the finite rational polyhedral decomposition locally exists. Before the proof, we need to first prepare a few auxiliary lemmas.

\begin{lemma} \label{lem:MMP K+D trivial in nbhd}
Let $\cP\subseteq V$ be a polytope such that $(X,\Delta)$ is log canonical for every $\Delta\in \cP$, and let $D\in \cP$ be such that $K_X+D$ is nef over $B$. Assume that $X$ is $\Q$-factorial. Then there exists a positive real number $\delta$ such that for any $G\in\cP$ satisfying $\norm{D-G}<\delta$, any $(K_X+G)$-MMP over $B$ is $(K_X+D)$-trivial.
\end{lemma}

\begin{proof}
This is a well known consequence of the lengths of extremal rays. See e.g. \cite{Bir11}*{Proposition 3.2(5)} or \cite{LT22}*{Proposition 2.12}.
\end{proof}

\begin{lemma} \label{lem:MMP in small nbhd give weak lc model}
In the setting of Theorem \ref{thm:chamber decomp by ample model}, assume that $X$ is $\Q$-factorial and let $D\in\cP$. Then there exists a positive real number $\delta$ such that: 

for any $G\in\cP$ satisfying $\norm{D-G}<\delta$, there exists a minimal model $X\dashrightarrow Y$ of $(X,G)$ over $B$ that is also a weak log canonical model of $(X,D)$ over $B$.
\end{lemma}

\begin{proof}
Since $(X, D)$ has a good minimal model over $B$, we can run a $(K_X+D)$-MMP $\varphi\colon X\dashrightarrow W$ over $B$ which terminates with a good minimal model $(W,D_W)$ of $(X, D)$ over $B$ by \cite{HH20}*{Theorem 1.7}. Since $X$ is $\Q$-factorial, $W$ is also $\Q$-factorial. Since $\varphi$ consists of only finitely many $(K_X+D)$-MMP steps over $B$, there exists a positive real number $\delta$ such that for any $G\in\cP$ satisfying $\norm{D-G}<\delta$, $\varphi$ is also a $(K_X+G)$-MMP over $B$. Since $(X, G)$ has a good minimal model over $B$, $(W, G_W)$ has a good minimal model over $B$ by Lemma \ref{lem:compare gmm existence on bir model}. By choosing $\delta$ sufficiently small, we can assume that for any $G\in\cP$ satisfying $\norm{D-G}<\delta$, any $(K_W+G_W)$-MMP over $B$ is $(K_W+D_W)$-trivial by Lemma \ref{lem:MMP K+D trivial in nbhd} since $K_W+D_W$ is nef over $B$. We can run a $(K_W+G_W)$-MMP over $B$ which terminates with a good minimal model $(Y,G_Y)$ of $(W, G_W)$ over $B$ by \cite{HH20}*{Theorem 1.7}. Since this $(K_W+G_W)$-MMP over $B$ is $(K_W+D_W)$-trivial, we see that $(Y,D_Y)$ is a weak log canonical model of $(X,D)$ over $B$.
\end{proof}

We also need a slight improvement of the above lemma that allows perturbation of the divisor $D$ to some $\Q$-divisor in the same neighbourhood. This is crucial in order to show that the polyhedral decomposition in Theorem \ref{thm:chamber decomp by ample model} can be made rational.

\begin{lemma} \label{lem:MMP in nbhd give weak lc model of Q-divisor}
In the setting of Theorem \ref{thm:chamber decomp by ample model}, assume that $X$ is $\Q$-factorial and let $D\in\cP$. Then there exists a positive real number $\delta$ such that:

for any $\Delta\in\cP$ in the rational envelope of $D$ satisfying $\norm{D-\Delta}<\delta$, and any $G\in\cP$ with $\norm{D-G}<\delta$, there exists a minimal model $X\dashrightarrow Y$ of $(X,G)$ over $B$ that is also a weak log canonical model of $(X,\Delta)$ over $B$.
\end{lemma}

Recall that the rational envelope of $D\in V$ is defined as the smallest rational affine subspace of $V$ that contains $D$.

\begin{proof}
Let $\delta$ be the positive real number given by Lemma \ref{lem:MMP in small nbhd give weak lc model}. We analyze the maps in the proof of Lemma \ref{lem:MMP in small nbhd give weak lc model}:
 \[
    \begin{tikzcd}
			(X,D) \arrow[r, dashed, "\varphi"] & (W,D_W) \arrow[r, dashed, "\psi"] & (Y,D_Y) . 
    \end{tikzcd}
\]
Recall that $\varphi$ is a $(K_X+G)$-MMP over $B$ for all $G\in \cP$ with $\norm{D-G}<\delta$ and $(Y,G_Y)$ is a minimal model of $(X,G)$ over $B$. Moreover, any $(K_W+G_W)$-MMP over $B$ is $(K_W+D_W)$-trivial. For any curve $C\subseteq W$, the condition $(K_W+D_W)\cdot C=0$ on $D$ defines a hyperplane in $V$ with rational coefficients, hence it contains the rational envelope of $D$. In particular, if $\Delta\in\cP$ lies in the rational envelope of $D$, then $(K_W+D_W)\cdot C=0$ implies $(K_W+\Delta_W)\cdot C=0$. Thus the $(K_W+G_W)$-MMP over $B$ is also $(K_W+\Delta_W)$-trivial. If in addition $\norm{D-\Delta}<\delta$, then we also know that $K_W+\Delta_W$ is nef over $B$: otherwise, the $(K_W+\Delta_W)$-MMP over $B$ contracts some curve $C\subseteq W$ with $(K_W+\Delta_W)\cdot C<0$, but since $\norm{D-\Delta}<\delta$, we also have $(K_W+D_W)\cdot C=0$, which implies $(K_W+\Delta_W)\cdot C=0$, a contradiction. Again because $\norm{D-\Delta}<\delta$, we have that $\varphi$ is a $(K_X+\Delta)$-MMP over $B$, thus $(W,\Delta_W)$ is also a minimal model of $(X,\Delta)$ over $B$. Since $\psi$ is $(K_W+\Delta_W)$-trivial, we get that $Y$ is a weak log canonical model of $(X,\Delta)$ over $B$. The proof is now complete.
\end{proof}

\begin{prop}\label{prop:local cone structure}
In the setting of Theorem \ref{thm:chamber decomp by ample model}, let $D\in\cP$. Then there exist a positive real number $\delta$ such that:

for any $\Delta\in\cP$ in the rational envelope of $D$ satisfying $\norm{D-\Delta}<\delta$, and any $G\in\cP$ with $\norm{D-G}<\delta$, the ample model of $(X, tG+(1-t)\Delta)$ over $B$ is independent of $t\in(0,1]$.
\end{prop}

\begin{proof}
First, we claim that we only need to show that there exists a positive real number $\delta$ such that for any $\Delta,G\in\cP$ as in the proposition, $(X, tG+(1-t)\Delta)$ has the same ample model over $B$ for every $t\in(0,1)$. Assume the existence of such a $\delta$. Since $\cP$ is a rational polytope and $\Delta\in\cP$ is in the rational envelope of $D$, possibly after replacing $\delta$ by some smaller $\delta'>0$, we may further assume that for any $G\in\cP$ satisfying $\norm{D-G}<\frac{\delta}{3}$ and any $\Delta\in\cP$ satisfying $\norm{D-\Delta}<\frac{\delta}{3}$, we have $G'=2G-\Delta\in\cP$ (in other words, $G=\frac{1}{2}(G'+\Delta)$) and $\norm{D-G'}<\delta$. Any convex combination $tG+(1-t)\Delta$ where $t\in (0,1]$ is also a convex combination $sG'+(1-s)\Delta$ for some $s\in (0,1)$ (in fact $s=\frac{t}{2}$). By the choice of $\delta$, we know that the ample model of $(X,sG'+(1-s)\Delta)$ over $B$ is independent of $s\in (0,1)$, hence the ample model of $(X, tG+(1-t)\Delta)$ over $B$ is independent of $t\in (0,1]$ whenever $\norm{D-\Delta},\norm{D-G}<\frac{\delta}{3}$.

We may also assume that $X$ is smooth. In fact, if $\pi\colon X'\to X$ is a log resolution of $(X,D_1+\dots+D_n)$ and let $E$ be the sum of the $\pi$-exceptional divisors, we may consider the rational polytope
\[\cP':=\{\Delta'=\Delta_{X'}+E\,|\,\Delta\in \cP\}.\]
For any $\Delta\in \cP$, since $(X,\Delta)$ is log canonical, we have $(X',\Delta')$ is also log canonical  where $\Delta'=\Delta_{X'}+E$ is as above and $K_{X'}+\Delta'\ge \pi^*(K_X+\Delta)$. Hence by Lemma \ref{lem:compare gmm existence on bir model}, $(X',\Delta')$ also has a good minimal model over $B$ and its ample model over $B$ is the same as the ample model of $(X,\Delta)$ over $B$. Thus in order to prove the proposition, we may replace $X$ (resp. $D$, resp. $\cP$) by $X'$ (resp. $D'$, resp. $\cP'$) and assume that $X$ is smooth.

We now proceed to prove that there exists a positive real number $\delta$ such that for any $\Delta,G\in\cP$ as in the proposition, $(X, tG+(1-t)\Delta)$ has the same ample model over $B$ for every $t\in(0,1)$. Let $\delta$ be the positive real number given by Lemma \ref{lem:MMP in nbhd give weak lc model of Q-divisor}. Then as long as $\norm{D-G},\norm{D-\Delta}<\delta$ and $\Delta$ is in the rational envelope of $D$, there exists a $(K_X+G)$-MMP $X\dashrightarrow Y$ over $B$ such that $Y$ is also a weak log canonical model of $(X,\Delta)$ over $B$. 
\[
    \begin{tikzcd}
       && Z\arrow[d, "i"]\\
			(X,\Delta) \arrow[r, dashed]  & (Y,\Delta_Y)  \arrow[r, "g"]\arrow[ur, "h"]& Z_0
    \end{tikzcd}
\]
Since $(X,\Delta)$ has a good minimal model over $B$, we can take the ample model $g\colon Y\to Z_0$ of $(Y,\Delta_Y)$ over $B$ such that
$$K_Y+\Delta_Y\sim_{\R}g^*H$$
where $H$ is an $\R$-Cartier $\R$-divisor on $Z_0$ that is ample over $B$. Similarly, we know that $K_Y+G_Y$ is semiample over $B$, hence it is also semiample over $Z_0$ and we take an ample model $h\colon Y\to Z$ of $(Y,G_Y)$ over $Z_0$ such that
$$K_Y+G_Y\sim_{\R}h^*H_Z$$
where $H_Z$ is an $\R$-Cartier $\R$-divisor on $Z$ that is ample over $Z_0$. Since $Y$ is a weak log canonical model of both $(X,\Delta)$ and $(X,G)$ over $B$, it is also a weak log canonical model of $(X,tG+(1-t)\Delta)$ over $B$ for any $t\in[0,1]$ by the definition of weak log canonical models. Since $H$ is ample over $B$ and $H_Z$ is ample over $Z_0$, for any $t\in (0,1)$ we may choose some sufficiently small $\varepsilon\in (0,t)$ such that $\varepsilon H_Z+(1-t)i^*H$ is ample over $B$. Note also that $H_Z$ is nef over $B$ since $K_Y+G_Y$ is nef over $B$. It follows that 
\begin{align*}
    K_Y+tG_Y+(1-t)\Delta_Y & = t(K_Y+G_Y)+(1-t)(K_Y+\Delta_Y) \\
    & \sim_{\R} h^*(tH_Z+(1-t)i^*H) \\
    & = h^*((t-\varepsilon)H_Z+(\varepsilon H_Z+(1-t)i^*H)) \\
    & = h^*(\mathrm{nef}/B+\mathrm{ample}/B) = h^*(\mathrm{ample}/B).
\end{align*}
Therefore, $h\colon Y\to Z$ is the common ample model of $(Y,tG_Y+(1-t)\Delta_Y)$ over $B$ for all $t\in(0,1)$ and we are done.
\end{proof}

\begin{proof}[Proof of Theorem \ref{thm:chamber decomp by ample model}]
Since $\cP$ is compact, we only need to prove the result locally around any $D\in\cP$. We proceed by induction on $\dim\cP$. The base case when $\dim\cP=0$ is clear. In the general case, by Proposition \ref{prop:local cone structure}, there exist a positive real number $\delta$ and a $\Q$-divisor $\Delta\in\cP$ with $\norm{D-\Delta}<\delta$ such that for any $G\in\cP$ satisfying $\norm{D-G}<\delta$, $(X, tG+(1-t)\Delta)$ has the same ample model over $B$ for every $t\in(0,1]$. Since we are allowed to work locally around $D\in\cP$, we only need to prove the result for any rational polytope $\cQ\subseteq\cP$ containing $D$ and $\Delta$ such that for any $Q\in\cQ$, $\norm{D-Q}<\delta$. For any $Q\in\cQ$ such that $Q\neq \Delta$, there exists $Q'\in\cQ$ such that $Q'$ is on the boundary of $\cQ$ and $Q$ is in the line segment connecting $Q'$ and $\Delta$. Since $\norm{D-Q'}<\delta$, $(X,Q)$ and $(X,Q')$ have the same ample model over $B$. Since the boundary $\partial \cQ$ of $\cQ$ consists of finitely many rational polytopes of lower dimension, we deduce from the induction hypothesis that there is a finite rational polyhedral decomposition of $\partial \cQ$ corresponding to the finitely many ample models of $(X,Q')$ over $B$ (as $Q'$ varies in $\partial \cQ$). As $\Delta$ is a $\Q$-divisor, we get an induced finite rational polyhedral decomposition of $\cQ$ by taking cones with vertex $\Delta$. By construction, the ample model over $B$ stays the same in each chamber.
\end{proof}

By the same argument as in this section, we can deduce a variant of Theorem \ref{thm:chamber decomp by ample model}. We omit the proof since it is almost the same as the proof of Theorem \ref{thm:chamber decomp by ample model}.

\begin{thm}\label{thm:chamber decomp by ample model irrational}
Notation as above. Let $\cP\subseteq V$ be a polytope such that $(X, D)$ is a log canonical pair and has a good minimal model over $B$ for every $D\in\cP$. 

Then there exist a finite polyhedral decomposition $\cP=\cup_{i=1}^p \cP_i$, together with finitely many rational maps $\psi_j\colon X\dashrightarrow Z_j$ over $B$, $1\le j\le q$, such that: for any $1\le i\le p$, there exists some $1\le j\le q$ such that $\psi_j\colon X\dashrightarrow Z_j$ is the ample model of $(X, D)$ over $B$ for any $D\in \cP_i^\circ$.
\end{thm}
 
\begin{rem}\label{rem:chamber decomp by ample model non-quasi-projective}
We remark that Theorems \ref{thm:chamber decomp by ample model} and \ref{thm:chamber decomp by ample model irrational} also work when $X$ and $B$ are normal varieties but not necessarily quasi-projective. We explain this briefly. We cover $B$ by finitely many quasi-projective open subsets $B_i$ and restrict the log canonical pairs given by the (rational) polytope $\cP$ to $X_i:=f^{-1}(B_i)$ for each $i$ where $X_i$ is quasi-projective (since $f$ is projective). We can apply Theorems \ref{thm:chamber decomp by ample model} and \ref{thm:chamber decomp by ample model irrational} to these pairs over $B_i$ and get a finite (rational) polyhedral decomposition of $\cP$ for each $i$. A common refinement of these decompositions gives the desired decomposition of $\cP$ regarding the log canonical pairs over $B$.
\end{rem}

\section{Wall-crossing morphisms}\label{sec:wall crossing}

In this section, we study the wall-crossing morphisms between (the reduced closure of) the images of the reduction morphisms and prove our main theorems. The results of this section rely on, among many others things, the MMP for locally stable families and the finiteness of ample models, see Sections \ref{sec:MMP results}, \ref{sec:MMP locally stable} and \ref{sec:Finiteness of ample models}.

We fix the following notation. Let $f\colon (X;D_1,\dots,D_n)\to B$ be an $n$-marked family over a reduced scheme. For any coefficient vector $\vv=(v_1,\dots,v_n)$, we define $\vv D:=\sum^{n}_{i=1} v_i D_i$. If $B$ is a normal variety and we let $\cP\subseteq\R_{\ge 0}^n$ be a polytope such that for every $\vb\in\cP$, the family $(X,\vb D)\to B$ is locally stable and its generic fiber has a semi-log-canonical model (see Definition \ref{defi:slc model}), then by Theorem \ref{thm:slc model exist}, for every $\vb\in \cP$, the family $(X,\vb D)\to B$ has an slc model $f_{\vb}\colon (Z_{\vb}, \vb\Delta_{\vb})\to B$ which is a stable family and thus induces a unique morphism $\Phi_{\vb}\colon B\to\cK_{\vb}$. We define $\cM_\vb = \overline{\Phi_\vb (B)}$ as the reduced closure of the image of $\Phi_\vb$ and $\cN_\vb$ as the normalization of $\cM_\vb$. 

We shall analyze the variation of $\cM_\vb$ as the coefficient vector $\vb$ varies in the polytope $\cP$. First, we prove an auxiliary lemma.

\begin{lemma}\label{lem:isomorphism descends to moduli spaces}
Let $(X;D_1,\dots,D_n)\to B$ be an $n$-marked family over a reduced proper scheme. Assume that for some $\vb,\vb'\in \R_{\ge 0}^n$, $(X,\vb D)\to B$ and $(X,\vb' D)\to B$ are stable families. Then $\cM_{\vb}\cong \cM_{\vb'}$.
\end{lemma}

\begin{proof}
We denote by $(\cX_{\vb},\vb\cD_{\vb})\to\cM_{\vb}$ the universal family over $\cM_{\vb}$. Since $B$ is proper and $\cM_{\vb}$ is separated, $B\to\cM_{\vb}$ is proper and surjective. Since $(X,\vb' D)\to B$ is a stable family and $B\to\cM_{\vb}$ is proper and surjective, it is not hard to check that $(\cX_{\vb},\vb'\cD_{\vb})\to\cM_{\vb}$ is a stable family. See the argument in \cite{ABIP23}*{Corollary 4.10}. Thus it induces a morphism $\cM_{\vb}\to\cM_{\vb'}$. Similarly, there also exists a morphism $\cM_{\vb'}\to\cM_{\vb}$. They are inverse to each other by construction.
\end{proof}

The following is the main result of this section.

\begin{thm} \label{thm:wall crossing general setup}
Let $f\colon (X;D_1,\dots,D_n)\to B$ be an $n$-marked family over a normal variety and let $\cP\subseteq\R_{\ge 0}^n$ be a polytope such that for every $\vb\in\cP$, the family $(X,\vb D)\to B$ is locally stable and its generic fiber has a semi-log-canonical model. Then
\begin{enumerate}
    \item For any $\va,\vb\in\cP$ such that the generic fiber of $(X,\va D)\to B$ is stable and $\va\ge \vb$, there exists a unique morphism $r_{\va,\vb}\colon \cN_\va\to \cN_\vb$ making the following diagram commute.
    \[
    \xymatrix{
		B \ar[dr] \ar[d] & \\
		\cN_\va \ar[r]^{r_{\va,\vb}}  & \cN_\vb 
        }
    \]

    \item There exists a finite polyhedral decomposition $\cP=\cup_{i=1}^p\cP_i$ such that for any $1\le i\le p$, the birational map $X\dashrightarrow Z_\vb$ and the underlying $n$-marked family $(Z_{\vb};\Delta_{\vb,1},\ldots,\Delta_{\vb,n})$ are independent of $\vb\in\cP_i^\circ$, and $\cM_\vb\cong \cM_{\vb'}$ for any $\vb,\vb'\in\cP_i^\circ$. Moreover, we can choose $\cP_i$ to be rational if $\cP$ is a rational polytope.
    
    \item For any $1\le i\le p$, we denote $\cM_\vb$ by $\cM_i$ for any $\vb\in \cP_i^\circ$. Assume that $f$ has a normal generic fiber. Then for any $1\le i,j\le p$ such that $\cP_j$ is a face of $\cP_i$, there exists a unique morphism $\alpha_{ij}\colon \cM_i\to \cM_j$ making the following diagram commute. Moreover, for any $1\le i,j,k\le p$ such that $\cP_j$ is a face of $\cP_i$ and $\cP_k$ is a face of $\cP_j$, $\alpha_{jk}\circ\alpha_{ij}=\alpha_{ik}$. The same statement holds for the normalization $\cN_i$ of $\cM_i$ without assuming that $f$ has a normal generic fiber.
    \begin{equation} \label{eq:wall crossing diagram}
    \xymatrix{
		B \ar[dr] \ar[d] & \\
		\cM_i \ar[r]^{\alpha_{ij}}  & \cM_j 
        }
    \end{equation}
\end{enumerate}
\end{thm}

We remark that when the base $B$ is proper, the statement is essentially a formal consequence of the results in the previous sections. The main missing part is the finite wall-and-chamber decomposition corresponding to $\cM_\vb$ when $B$ is not proper. We will reduce this to the proper case by applying weak semistable reduction.

\begin{proof}
Since $B$ is normal, the morphism $B\to\cM_{\vb}$ lifts to a unique morphism $B\to\cN_{\vb}$ for every $\vb\in\cP$. Note that $B$ dominates $\cN_{\vb}$. Since we also have that $\cN_{\va}$ and $\cN_{\vb}$ are normal and separated, the morphism $r_{\va,\vb}$ is unique. Let $(\cX_{\cN_{\va}},\va\cD_{\cN_{\va}})\to\cN_{\va}$ be the universal family over $\cN_{\va}$. Since the generic fiber of $(X,\va D)\to B$ is stable, the pullback of $(\cX_{\cN_{\va}},\va\cD_{\cN_{\va}})\to\cN_{\va}$ via $B\to\cN_{\va}$ is isomorphic to $(X,\va D)\to B$ over the generic point of $B$. Since $\cN_{\va}$ is normal and the generic fiber of $(X,\vb D)\to B$ has an slc model, $(\cX_{\cN_{\va}},\vb\cD_{\cN_{\va}})\to\cN_{\va}$ has an slc model by Theorem \ref{thm:slc model exist} which induces a unique morphism $r_{\va,\vb}\colon \cN_\va\to\cN_\vb$. The composition of $r_{\va,\vb}$ and $B\to\cN_{\va}$ is clearly $B\to\cN_\vb$. This proves statement (1).

For the statement (2), we divide its proof in two parts. First we construct a polyhedral decomposition corresponding to the birational map $X\dashrightarrow Z_\vb$ to the slc model. Observe that we may assume that $B$ is smooth. To see this, let $\tB\to B$ be a resolution of singularities. For every $\vb\in\cP$, let $(\tX,\vb\tD)\to \tB$ be the locally stable family obtained by base change. It has an slc model $\tf_{\vb}\colon (\tZ_{\vb}, \vb\tDelta_{\vb})\to \tB$ by Theorem \ref{thm:slc model exist} and the assumptions. As in the proof of Theorem \ref{thm:slc model exist}, $\tf_{\vb}\colon (\tZ_{\vb}, \vb\tDelta_{\vb})\to \tB$ descends to $f_{\vb}\colon (Z_{\vb}, \vb\Delta_{\vb})\to B$ for every $\vb\in\cP$. If there exists a finite (rational) polyhedral decomposition $\cP=\cup_{i=1}^p\cP_i$ such that for any $1\le i\le p$, the birational map $\tX\dashrightarrow \tZ_\vb$ is independent of $\vb\in\cP_i^\circ$, then the birational map $X\dashrightarrow Z_\vb$ is also independent of $\vb\in\cP_i^\circ$. Thus we may assume $B$ is smooth.

Since $B$ is smooth, we deduce that $(X,\vb D)$ is slc for every $\vb\in\cP$ by Theorem \ref{thm:locally stable smooth base}. Let $(X^\nu,\Gamma+\vb D^\nu)\to B$ be the normalization of $(X,\vb D)\to B$ where $\Gamma$ is the conductor and $D^\nu$ is the strict transform of $D$. Note that the induced involution $\tau$ on $\Gamma^\nu$ is determined by the normalization $X^\nu\to X$ and is independent of $\vb\in\cP$. By definition, the normalization $(Z^\nu_{\vb},\Gamma_{\vb}+\vb\Delta^\nu_{\vb})$ of $(Z_{\vb},\vb \Delta_{\vb})$ is the log canonical model of $(X^\nu,\Gamma+\vb D^\nu)$ over $B$. By applying Theorems \ref{thm:chamber decomp by ample model} and \ref{thm:chamber decomp by ample model irrational} to the log canonical pairs $(X^\nu,\Gamma+\vb D^\nu)$ over $B$ where $\vb$ varies in $\cP$, we get a finite (rational) polyhedral decomposition $\cP=\cup_{i=1}^p\cP_i$ such that for any $1\le i\le p$, the birational contraction $X^\nu\dashrightarrow Z^\nu_{\vb}$ is independent of $\vb\in\cP_i^\circ$. Thus the induced regular involution $\tau_\vb$ on the normalization $\Gamma^\nu_{\vb}$ of the conductor $\Gamma_{\vb}$ is also independent of $\vb\in\cP_i^\circ$. By \cite{Kol13}*{Proposition 5.3}, we conclude that for any $1\le i\le p$, the birational map $X\dashrightarrow Z_\vb$ is independent of $\vb\in\cP_i^\circ$. Note that this also implies that the underlying $n$-marked family $(Z_{\vb};\Delta_{\vb,1},\ldots,\Delta_{\vb,n})$ is independent of $\vb\in\cP_i^\circ$. In particular, if $B$ is also proper, then by Lemma \ref{lem:isomorphism descends to moduli spaces}, we obtain $\cM_\vb \cong \cM_{\vb'}$ for any $\vb,\vb'\in\cP_i^\circ$ and therefore statement (2) holds in this case.

In general, we prove that there exists a finite (rational) polyhedral decomposition $\cP=\cup_{i=1}^q\cQ_i$ such that for any $1\le i\le q$, $\cM_\vb\cong \cM_{\vb'}$ for any $\vb,\vb'\in\cQ_i^\circ$. We will reduce it to the case when the base is proper. We may freely shrink $B$ as needed since it will not change $\cM_\vb$. In particular, we may assume that $B$ is smooth. We continue to use the same notation as in the paragraph above. Let $T:=(\sum^{n}_{i=1} D_i)_{\red}$. Let $Y\to X^\nu$ be a log resolution of $(X^\nu,\Gamma+T^\nu)$ and $E_Y$ the sum of its exceptional divisors. Then $(Y,\Gamma_Y+T^\nu_Y+E_Y)$ is log smooth. By shrinking $B$, we can assume that $(Y,\Gamma_Y+T^\nu_Y+E_Y)$ is log smooth over $B$, and thus $(Y,\Gamma_Y+T^\nu_Y+E_Y)\to B$ is locally stable (note that $\Gamma_Y+T^\nu_Y+E_Y$ is a reduced divisor). By Abramovich--Karu's weak semistable reduction theorem \cite{AK00} as used in the proof of \cite{Kol23}*{Theorem 4.59}, there exist a projective generically finite surjective morphism $V\to B$ and a compactification $V\to\oV$ such that the pullback of $(Y,\Gamma_Y+T^\nu_Y+E_Y)\to B$ via $V\to B$ extends to a locally stable family over $\oV$ (in fact one can make $\oV$ smooth and the family toroidal, but we will not need this fact here). We can assume that $V\to B$ is \'etale by shrinking $B$ further. Then the formation of the slc model of $(X,\vb D)\to B$ commutes with the base change via $V\to B$ for every $\vb\in\cP$. Thus $\cM_\vb$ will not change if we replace $B$ by $V$ and pull back everything to $V$. Normalize $\oV$ if necessary, we may assume that $B$ has a normal compactification $\oB$ such that $(Y,\Gamma_Y+T^\nu_Y+E_Y)\to B$ extends to a locally stable family $(\oY,\Gamma_{\oY}+T^\nu_{\oY}+E_{\oY})\to \oB$ with normal generic fiber. By applying Lemma \ref{lem:small Q-Cartier modification over normal base} finitely many times, we can assume that $K_{\oY/\oB}+\Gamma_{\oY}+\vb D^\nu_{\oY}+E_{\oY}$ is $\R$-Cartier for every $\vb\in\cP$. Since $(X,\vb D)\to B$ is locally stable, we have $\vb D\le T$ and thus $\vb D^\nu_{\oY}\le T^\nu_{\oY}$. Thus $(\oY,\Gamma_{\oY}+\vb D^\nu_{\oY}+E_{\oY})\to \oB$ is locally stable for every $\vb\in\cP$. The situation is partly summarized in the following diagram.
\[
    \xymatrix@C+2pc{
         (Y,\Gamma_Y+\vb D^\nu_Y+E_Y) \ar[r]^-{\txt{\tiny log}}_-{\txt{\tiny resolution}} & (X^\nu,\Gamma+\vb D^\nu) \ar[r]^-{\txt{\tiny normalization}} & (X,\vb D)
		}
\]
Since $(X^\nu,\Gamma+\vb D^\nu)$ is log canonical by Theorem \ref{thm:locally stable smooth base}, the log canonical model of $(Y,\Gamma_Y+\vb D^\nu_Y+E_Y)$ over $B$ is the same as that of $(X^\nu,\Gamma+\vb D^\nu)$ by Lemma \ref{lem:compare gmm existence on bir model}, the latter being $(Z^\nu_{\vb}, \Gamma_{\vb}+\vb\Delta^\nu_{\vb})$. Thus $(\oY,\Gamma_{\oY}+\vb D^\nu_{\oY}+E_{\oY})$ also has a log canonical model over $\oB$ by Theorem \ref{thm:slc model exist}, and it is a stable family over $\oB$ (by definition) that extends the stable family $(Z^\nu_{\vb}, \Gamma_{\vb}+\vb\Delta^\nu_{\vb})\to B$. Denote by $(\oZ^\nu_{\vb},\oGamma_{\vb}+\vb\oDelta^\nu_{\vb})\to\oB$ the extended family. By Corollary \ref{cor:extend family}, the stable family $(Z_{\vb},\vb \Delta_{\vb})\to B$ also extends to some stable family $(\oZ_{\vb},\vb \oDelta_{\vb})\to \oB$, and the normalization of $\oZ_{\vb}$ is exactly $\oZ^\nu_{\vb}$. Since $\oB$ is proper, from the already established case we see that there is a finite (rational) polyhedral decomposition $\cP=\cup_{i=1}^q\cQ_i$ such that for any $1\le i\le q$, the birational contraction $\oY\dashrightarrow\oZ^\nu_{\vb}$ is independent of $\vb\in\cQ_i^\circ$. In particular, the birational involution $\tau_\vb$ on the normalization $\oGamma^\nu_{\vb}$ of the conductor $\oGamma_{\vb}$ is also independent of $\vb\in\cQ_i^\circ$. By \cite{Kol13}*{Proposition 5.3}, this implies that the birational map $\oY\dashrightarrow \oZ_{\vb}$ is independent of $\vb\in\cQ_i^\circ$ as well. Thus the underlying $n$-marked family $(\oZ_{\vb};\oDelta_{\vb,1},\ldots,\oDelta_{\vb,n})$ is independent of $\vb\in\cQ_i^\circ$ for any $1\le i\le q$. On the other hand, the stable family $(\oZ_{\vb}, \vb\oDelta_{\vb})\to\oB$ induces a unique morphism $\oPhi_{\vb}\colon\oB\to\cK_{\vb}$ that extends $\Phi_{\vb}\colon B\to\cK_{\vb}$. Thus $\cM_\vb$ is the reduced image of $\oPhi_{\vb}$ since $\oB$ is proper. By Lemma \ref{lem:isomorphism descends to moduli spaces}, we get $\cM_{\vb}\cong \cM_{\vb'}$ for any $1\le i\le q$ and any $\vb,\vb'\in\cQ_i^\circ$. A common refinement of the decompositions $\cP=\cup_{i=1}^p\cP_i$ and $\cP=\cup_{i=1}^q\cQ_i$ gives the desired finite (rational) polyhedral decomposition of $\cP$ in statement (2).

Finally, we prove statement (3). For any $1\le i\le p$, we denote the underlying $n$-marked families $(\cX_{\vb},\cD_{\vb})\to\cM_{\vb}$ by $(\cX_i,\cD_i)\to\cM_i$, $(\cX_{\cN_{\vb}},\cD_{\cN_{\vb}})\to\cN_{\vb}$ by $(\cX_{\cN_i},\cD_{\cN_i})\to\cN_i$ and $(Z_{\vb}, \Delta_{\vb})\to B$ by $(Z_i, \Delta_i)\to B$ for any $\vb\in\cP_i^\circ$. This notation is well-defined by statement (2). Since $\cM_i$ and $\cM_j$ are reduced and separated and $B$ dominates $\cM_i$, we see that $\alpha_{ij}$ is unique and the diagram \eqref{eq:wall crossing diagram} commutes if it commutes over an open set of $B$. Thus to construct $\alpha_{ij}$, we may shrink $B$ and assume that $B$ is smooth. We first prove statement (3) for $\cM_i$ assuming that the generic fiber of $f$ is normal. We deduce that $X$ is normal. Since $(Z_i, \vb\Delta_i)$ is the log canonical model of $(X,\vb D)$ over $B$ for every $\vb\in\cP_i^\circ$, $(Z_i, \vb\Delta_i)$ is a weak log canonical model of $(X,\vb D)$ over $B$ for every $\vb\in\cP_i$ by definition. Since $(X,\vb D)$ has a good minimal model over $B$, we deduce that $K_{Z_i/B}+\vb\Delta_i$ is semiample over $B$ for every $\vb\in\cP_i$, hence $K_{\cX_{i}/\cM_i}+\vb\cD_i$ is semiample over the generic point of $\cM_i$. Since $(\cX_{i},\vb\cD_{i})\to\cM_{i}$ is stable for every $\vb\in\cP_i^\circ$, we see that $(\cX_{i},\vb\cD_{i})\to\cM_{i}$ is locally stable, and $K_{\cX_{i}/\cM_i}+\vb\cD_i$ is nef over $\cM_{i}$ for every $\vb\in\cP_i$. By Lemma \ref{lem:ample model commutes with base change}, this further implies that $K_{\cX_{i}/\cM_i}+\vb\cD_i$ is semiample over $\cM_{i}$. Let $\vb_j$ be any point in $\cP_j^\circ\subseteq\cP_i$. By Lemma \ref{lem:ample model stable}, the log canonical model of $(\cX_{i},\vb_j\cD_{i})\to\cM_{i}$ is a stable family and induces a morphism $\alpha_{ij}\colon\cM_i\to\cM_j$. By construction and Lemma \ref{lem:ample model commutes with base change}, the composition of $\alpha_{ij}$ and $B\to\cM_i$ is $B\to\cM_j$. For any $1\le i,j,k\le p$ such that $\cP_j$ is a face of $\cP_i$ and $\cP_k$ is a face of $\cP_j$, the compositions of $\alpha_{ik}$ or $\alpha_{jk}\circ\alpha_{ij}$ with $B\to\cM_i$ both equal $B\to\cM_k$, thus the uniqueness of the diagram \eqref{eq:wall crossing diagram} implies that $\alpha_{jk}\circ\alpha_{ij}=\alpha_{ik}$.

Next, without assuming that the generic fiber of $f$ is normal, we prove statement (3) for $\cN_i$. We only need to prove that for any $1\le i,j\le p$ such that $\cP_j$ is a face of $\cP_i$, there exists a morphism $\beta_{ij}\colon \cN_i\to \cN_j$ whose composition with $B\to\cN_i$ is $B\to\cN_j$ since the rest of the statement follows from the same argument. Similarly as above, we can assume that $B$ is smooth. For every $\vb\in\cP_i^\circ$, $(Z_i, \vb\Delta_i)\to B$ is the slc model of $(X,\vb D)\to B$ and thus its normalization $(Z^\nu_{i},\Gamma_{i}+\vb\Delta^\nu_{i})$ is the log canonical model of $(X^\nu,\Gamma+\vb D^\nu)$ over $B$. By the same argument as above, for every $\vb\in\cP_i$, the log canonical model of $(Z^\nu_{i},\Gamma_{i}+\vb\Delta^\nu_{i})$ over $B$ exists and it is the same as that of $(X^\nu,\Gamma+\vb D^\nu)$. Let $\vb_j$ be any point in $\cP_j^\circ\subseteq\cP_i$. We deduce that the slc model $(Z_j, \vb_j\Delta_j)\to B$ of $(X,\vb_j D)\to B$ is the same as that of $(Z_i, \vb_j\Delta_i)\to B$ (see Definition \ref{defi:slc model} and Theorem \ref{thm:slc model exist}). The pullback of $(\cX_{\cN_i},\vb_j\cD_{\cN_i})\to\cN_i$ via $B\to\cN_i$ is isomorphic to $(Z_i, \vb_j\Delta_i)\to B$. Since $\cN_i$ is normal and the generic fiber of $(Z_i, \vb_j\Delta_i)\to B$ has an slc model, $(\cX_{\cN_i},\vb_j\cD_{\cN_i})\to\cN_i$ has an slc model by Theorem \ref{thm:slc model exist} which induces a unique morphism $\beta_{ij}\colon \cN_i\to \cN_j$. The composition of $\beta_{ij}$ and $B\to\cN_i$ is $B\to\cN_j$ by the discussion above.
\end{proof}

We remark that the morphism $\alpha_{ij}$ in the theorem above is given by taking fiberwise ample model by the construction of $\alpha_{ij}$ and Lemma \ref{lem:ample model commutes with base change}.

\subsection{Proof of main theorems}

We are now in a position to prove our main theorems.

\begin{proof}[Proof of Theorem \ref{main thm:weight reduction}]
The generic fiber of $(X,\va D)\to B$ has a good minimal model by Lemma \ref{lem:compare gmm existence on bir model} since it has a log canonical model. Since each irreducible component of $D$ dominates $B$ and $\mathbf{0}<\vb\le\va$, the restrictions of $\va D$ and $\vb D$ to the generic fiber of $X\to B$ have the same support. By Theorem \ref{thm:lc model exists non-R-Cartier}, the generic fiber of $(X,\vb D)\to B$ has a log canonical model since $K_{X/B}+\vb D$ is big over $B$. By Theorem \ref{thm:slc model exist}, we finish the proof.
\end{proof}

\begin{proof}[Proof of Theorem \ref{main thm:wall crossing}]
By applying Lemma \ref{lem:small Q-Cartier modification over normal base} finitely many times, we get a small birational projective morphism $\pi\colon Y\to X$ such that $K_{Y/B}$ and each irreducible component of $D_Y$ are $\Q$-Cartier and $(Y,\va D_Y)\to B$ is locally stable with normal generic fiber. Thus $(Y,\vb D_Y)\to B$ is locally stable for every $\vb\in\cP$ since $\vb\le\va$ and $K_{Y/B}+\vb D_Y$ is $\R$-Cartier. The log canonical model of $(Y,\vb D_Y)\to B$ exists by Theorem \ref{main thm:weight reduction} and coincides with that of $(X,\vb D)\to B$ by Remark \ref{rem:equivalent defn of lc model for family}. By applying Theorem \ref{thm:wall crossing general setup}(2)(3) to the locally stable families $(Y,\vb D_Y)\to B$ where $\vb$ varies in $\cP$, we finish the proof.
\end{proof}

If the family $(X,\va D)\to B$ is stable, its generic fiber has a log canonical model by definition. Thus one can apply Theorems \ref{main thm:weight reduction} and \ref{main thm:wall crossing} to this special case.

\begin{proof}[Proof of Theorem \ref{main thm:slc case}]
By Theorem \ref{thm:wall crossing general setup}, the reduction morphism $r_{\va,\vb}$ exists for any $\va \ge \vb$ in $\cP$, and there is a wall-and-chamber and wall-crossing structure on $\cP$ corresponding to $\cN_\va$. By the same argument as in \cite{ABIP23}*{Proposition 7.7}, $r_{\va,\vb}$ is birational. Since $(X,\va D)\to B$ is stable for every $\va\in \cP$, $r_{\vb,\vc}\circ r_{\va,\vb}$ and $r_{\va,\vc}$ coincide on a dense open substack of $\cN_{\va}$. Thus $r_{\vb,\vc}\circ r_{\va,\vb} = r_{\va,\vc}$ since $\cN_{\va}$ and $\cN_{\vc}$ are normal and separated.
\end{proof}

\begin{prop}
In the setting of Theorem \ref{main thm:wall crossing}, there exists a nonempty open subset $U\subseteq B$ which is independent of $\vb\in\cP$ such that $\Phi_{\vb}(u)$ is the point classifying the log canonical model of $(X_u,\vb D_u)$ for every $u\in U$ and every $\vb\in\cP$.
\end{prop}

\begin{proof}
Recall that $\Phi_{\vb}\colon B\to\cK_{\vb}$ is the unique morphism induced by the stable family $f_{\vb}\colon (Z_{\vb}, \vb\Delta_{\vb})\to B$. By shrinking $B$, we can assume that $B$ is smooth and quasi-projective. By Theorem \ref{thm:locally stable smooth base}, $(X,\va D)$ is log canonical. Let $g\colon Y\to X$ be a log resolution of $(X,\va D)$ and $E$ the sum of the $g$-exceptional divisors. Then $(Y,\va D_Y+E)$ is log smooth. By Remark \ref{rem:different ways to define log canonical models are the same}, $f_{\vb}\colon (Z_{\vb}, \vb\Delta_{\vb})\to B$ is also the log canonical model of $(Y,\vb D_Y+E)$ over $B$. We can choose an affine open subset $U\subseteq B$ such that $(Y,\va D_Y+E)$ is log smooth over $U$. By shrinking $U$, we can also assume that $(X_u,\va D_u)$ is log canonical for every $u\in U$ since $(X,\va D)$ is log canonical. We deduce that the log canonical model of $(X_u,\vb D_u)$ is the same as that of $(Y_u,\vb (D_Y)_u+E_u)$ for every $u\in U$ and every $\vb\in\cP$ by Remark \ref{rem:different ways to define log canonical models are the same}. Thus we only need to prove that the log canonical model of $(Y_u,\vb (D_Y)_u+E_u)$ is the same as the fiber of the log canonical model of $(Y,\vb D_Y+E)$ over $B$ at the point $u\in U$ for every $u\in U$ and every $\vb\in\cP$. By the same argument as in Lemma \ref{lem:semiample reduce to Q-coef}, we can assume that $\vb$ has rational coordinates. By \cite{HMX18}*{Corollary 1.4}, we only need to prove that there exists a closed point $0\in U$ such that $(Y_0,\vb (D_Y)_0+E_0)$ has a good minimal model. Let $W=(f\circ g)^{-1}(U)$. By \cite{HMX18}*{Lemma 6.2}, it suffices to prove that $(W,\vb D_Y|_W+E|_W)$ has a good minimal model. Since $U$ is affine, $(W,\vb D_Y|_W+E|_W)$ has a good minimal model if it has a good minimal model over $U$. Thus we only need to prove that $(Y,\vb D_Y+E)$ has a good minimal model over $B$ for every $\vb\in\cP$. Since $(X,\va D)$ has a log canonical model over $B$ by Theorem \ref{thm:MMP locally stable family}, $(Y,\va D_Y+E)$ also has one and thus it has a good minimal model over $B$ by Lemma \ref{lem:compare gmm existence on bir model}. Since $\mathbf{0}<\vb\le\va$, $\Supp(\vb D_Y+E)=\Supp(\va D_Y+E)$. By Theorem \ref{thm:gmm smaller coef}, $(Y,\vb D_Y+E)$ has a good minimal model over $B$ for every $\vb\in\cP$.
\end{proof}

\begin{rem}
In addition to weight reduction morphisms, \cite{Has03} also studies forgetting morphisms (which is equivalent to reducing certain weights to zero) between moduli spaces $\cM_{g,\va}$ of curves with marked points. For moduli of stable pairs, if we restrict to the components whose generic points parameterize log canonical pairs, then the existence of forgetting morphisms is essentially equivalent to the abundance conjecture. In fact, if abundance holds, then the assumption of Theorem \ref{thm:wall crossing general setup} is automatically satisfied and hence we get the forgetting morphisms. On the other hand, if $(X,D)$ is a log canonical pair such that $K_X+D$ is nef and big and $H\ge 0$ is some general ample divisor on $X$, then $(X,D+\varepsilon H)$ is a stable pair for any $0<\varepsilon\ll 1$. To construct the forgetting morphism when $\varepsilon\to 0$ requires taking the ample model of $K_X+D$, and it is well known that abundance for log canonical pairs of log general type is equivalent to abundance in general.
\end{rem}

\section{Examples} \label{sec:examples}

In this section, we give some examples to show that some of the assumptions in our main results are optimal.

\subsection{Reducing weights on non-normal stable pairs} \label{ss:slc counterexample}

Since MMP can fail for semi-log-canonical pairs, we cannot reduce coefficients in an arbitrary way if the locally stable family has a non-normal generic fiber. The following is an explicit example.

\begin{exa}
Take $X_1=\bP^1\times\bP^1$ and let $X_2$ be the blowup of a point on $\bP^1\times\bP^1$. Let $E\subseteq X_2$ be the exceptional curve. On both $X_i$, let $F$ be the divisor class of $\{\mathrm{point}\}\times \bP^1$ and let $\Gamma$ be the divisor class of $\bP^1\times \{\mathrm{point}\}$. Our surface $X$ is obtained by gluing together $X_1$ and $X_2$ along some curves $\Gamma_i\subseteq X_i$ ($i=1,2$) where $\Gamma_1$ is a general member of the linear system $|\Gamma|$ and $\Gamma_2$ is the unique member of $|\Gamma-E|$. We need to also specify the boundary divisor $D$. On $X_1$ we take $D$ as the sum of two general members of $|\Gamma|$ together with three general members of $|F|$. On $X_2$ we take $D$ as the sum of two general members $D_1,D_2$ of the (base point free) linear system $|\Gamma+F-E|$ plus $E$ plus two general members of $|\Gamma+F|$. The following picture is an illustration of the resulting configuration. 
\medskip
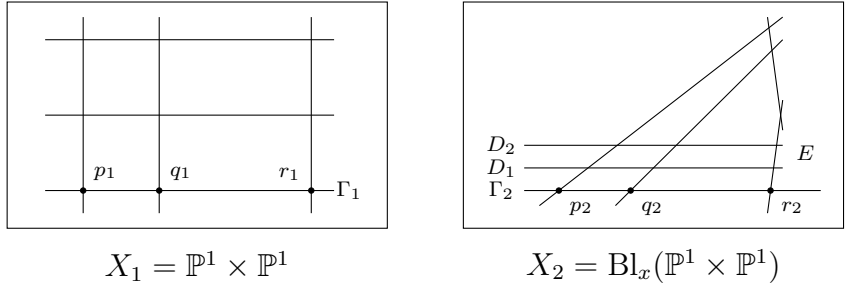
\begin{center}
    \begin{tikzpicture}
    \begin{scope} % first surface
        \node at (2.5,-0.5) {$X_1=\bP^1\times\bP^1$};
        \draw (0,0) rectangle (5,3);
        \draw[name path=Gamma1] (0.5,0.5) -- (4.3,0.5);
        \node[font=\tiny] at (4.5,0.5) {$\Gamma_1$};
        \draw (0.5,1.5) -- (4.3,1.5);
        \draw (0.5,2.5) -- (4.3,2.5);
        \draw[name path=Fp] (1,0.2) -- (1,2.8);
        \draw[name path=Fq] (2,0.2) -- (2,2.8);
        \draw[name path=Fr] (4,0.2) -- (4,2.8);
        \path [name intersections={of=Fp and Gamma1,by=p1}];
        \path [name intersections={of=Fq and Gamma1,by=q1}];
        \path [name intersections={of=Fr and Gamma1,by=r1}];
        
        \filldraw[black] (p1) circle (1pt) node[anchor=south west,font=\tiny] {$p_1$};
        \filldraw[black] (q1) circle (1pt) node[anchor=south west,font=\tiny] {$q_1$};
        \filldraw[black] (r1) circle (1pt) node[anchor=south east,font=\tiny] {$r_1$};
    \end{scope}
    
    \begin{scope} % second surface
        \node at (8.5,-0.5) {$X_2=\mathrm{Bl}_x (\bP^1\times\bP^1)$};
        \draw (6,0) rectangle (11,3);
        
        \draw[name path=F0] (10,2.8) -- (10.2,1.3);
        %\node[font=\tiny] at (10.5,2) {$F_0$};
        
        \draw[name path=E] (10,0.2) -- (10.2,1.7);
        \node[font=\tiny] at (10.5,1) {$E$};
        
        \draw[name path=Gamma2] (6.8,0.5) -- (10.7,0.5);
        \node[font=\tiny] at (6.5,0.5) {$\Gamma_2$};
        \path [name intersections={of=E and Gamma2,by=r2}];
        \filldraw[black] (r2) circle (1pt) node[anchor=north west,font=\tiny] {$r_2$};
        
        \draw (6.8,0.8) -- (10.2,0.8);
        \draw (6.8,1.1) -- (10.2,1.1);
        \node[font=\tiny] at (6.5,0.8) {$D_1$};
        \node[font=\tiny] at (6.5,1.1) {$D_2$};

        \draw[name path=Dp] (7,0.3) -- (10.2,2.8);
        \draw[name path=Dq] (8,0.3) -- (10.2,2.5);
        \path [name intersections={of=Dp and Gamma2,by=p2}];
        \path [name intersections={of=Dq and Gamma2,by=q2}];
        \filldraw[black] (p2) circle (1pt) node[anchor=north west,font=\tiny] {$p_2$};
        \filldraw[black] (q2) circle (1pt) node[anchor=north west,font=\tiny] {$q_2$};
    \end{scope}
    \end{tikzpicture}
\end{center}
On each component $X_i$ the divisor $D$ intersects $\Gamma_i$ in three points $p_i,q_i,r_i$, and the gluing between $X_1$ and $X_2$ should identify $\{p_1,q_1,r_1\}$ with $\{p_2,q_2,r_2\}$. It is not hard to check that $K_{X_i}+D|_{X_i}+\Gamma_i$ is ample on each component, hence by Koll\'ar's gluing theory \cite{Kol13}*{Theorem 5.13}, we know that $(X,D)$ is a stable pair.

Note that $D_1$ and $D_2$ do not intersect the double locus of $X$, hence they are Cartier divisors on $X$. If we drop the coefficients of $D_1$ and $D_2$ in $D$ to some $c\in (0,\frac{1}{2})$ and denote the new boundary divisor by $D'$, then $K_{X_1}+D'|_{X_1}+\Gamma_1$ remains unchanged, while $K_{X_2}+D'|_{X_2}+\Gamma_2$ has negative intersection with $E$ (we have $(K_{X_2}+D|_{X_2}+\Gamma_2)\cdot E=D_1\cdot E=D_2\cdot E=1$, hence $(K_{X_2}+D'|_{X_2}+\Gamma_2)\cdot E=2c-1$), thus the MMP $\pi\colon X_2\to X'_2$ on $X_2$ contracts the curve $E$ and we get $X'_2=\bP^1\times \bP^1$. However, the pair $(X'_2,\pi_*(D'|_{X_2}))$ no longer glues with $(X_1,D'|_{X_1})$, since $(\pi_*(D'|_{X_2})\cdot \pi_*\Gamma_2)=2+2c<3=(D'|_{X_1}\cdot \Gamma_1)$, i.e. the differents along the conductors do not match.
\end{exa}

\subsection{Reduction morphism does not factor}

By modifying an example in \cite{ABIP23}*{Section 8.3}, we show that in Theorem \ref{thm:wall crossing general setup} the reduction morphism $B\to \cN_\vb$ does not factor through $\cN_\va$ if we do not assume that the generic fiber of $(X,\va D)\to B$ is stable (in other words, if $K_{X/B}+\va D$ is only relatively big but not ample along the generic fiber). Our examples have klt generic fibers, so they also illustrate that the ampleness assumption in the main results of \cite{ABIP23} is necessary (more specifically, the only place in \emph{loc. cit.} that relies on this assumption is their Theorem 7.6).

\begin{exa}
Consider a line $L_0$ in $\bP^2$ and a smooth plane curve $C_0\subseteq \bP^2$ of degree $d\gg 0$ that has transversal intersection with $L_0$. Let $p\in C_0\cap L_0$, let $\pi\colon X\to \bP^2$ be the blow up at $p$, let $C$ (resp. $L$) be the strict transform of $C_0$ (resp. $L_0$), and let $E$ be the $\pi$-exceptional curve. Choose some $a\in (\frac{1}{2},1)$. We have
\[K_X+a(C+L)+(2a-1)E=\pi^*(K_{\bP^2}+a(C_0+L_0)).\]
Thus if $b\ge 2a-1$, the ample model of $(X,a(C+L)+bE)$ is $(\bP^2,a(C_0+L_0))$, while for $b\in (2a-1-\varepsilon,2a-1)$ (where $0<\varepsilon\ll 1$), the ample model of $(X,a(C+L)+bE)$ is itself. Now consider the moduli space $U$ of such (marked) family $(X;C,L,E)$ as $C_0,L_0$ and $p\in C_0\cap L_0$ vary. Note that for a fixed choice of $C_0$ and $L_0$ there are exactly $d$ possible ways to choose $p$. Moreover, since $E$ is the only $(-1)$-curve on $X$, any isomorphism $(X,C+L)\cong (X',C'+L')$ necessarily comes from an isomorphism $(\bP^2,C_0+L_0)\cong (\bP^2,C'_0+L'_0)$ sending $p$ to $p'$. For general choice of $C_0$ and $L_0$, the pair $(\bP^2,C_0+L_0)$ does not have any non-trivial automorphisms, hence different choices of the intersection point $p$ give non-isomorphic $(X,C+L)$. Putting the above discussions together, we see that the reduction morphism $U\to \cN_{a,b}$ is birational when $2a-1-\varepsilon<b<2a-1$, and is generically finite of degree $d>1$ when $b\ge 2a-1$. It follows that the natural (wall-crossing) map $\cN_{a,2a-1-\varepsilon}\to \cN_{a,2a-1}$ is also generically finite of degree $d>1$ and the reduction morphism $U\to \cN_{a,2a-1-\varepsilon}$ does not factor through $\cN_{a,2a-1}$. 
\end{exa}

\subsection{Walls for \texorpdfstring{$\cK_\va$}{moduli} when coefficients vary}

In general, the moduli 
\[\cK_{\va,\le v} := \cup_{v_0\in [0,v]} \cK_{\va,v_0}\]
of stable pairs with coefficient vector $\va$ and volume at most $v$ may have infinitely many connected components. A priori, we could still investigate whether there is some finite wall-and-chamber structure in this setting. We declare that a coefficient vector $\va$ sits on a wall if there exist a pair $(X,D)$ and some $\va_i$ $(i=1,2,\dots)$ converging to $\va$ such that $(X,\va_i D)$ are stable pairs of volume at most $v$ for all $i$ but $(X,\va D)$ is not. The next series of examples show that such walls are often badly behaved.

\begin{exa}
Let $C_d\subseteq\bP^2$ be a smooth plane curve of degree $d\ge 4$, then $(\bP^2,aC_d)$ is stable if and only if $a>\frac{3}{d}$. Thus we hit a wall at $a=\frac{3}{d}$. As $d$ varies, we get infinitely many walls near $0$.
\end{exa}

\begin{exa}
Fix an integer $d\ge 2$. For each $m\in\bN^*$, let $C_{m}$ be a general member of the base point free linear system $|\cO(dm)|$ on $\bP(1,1,m)$. Then $(\bP(1,1,m),aC_m)$ is stable if and only if $a>\frac{1}{d}+\frac{2}{dm}$. Thus there are walls at $\frac{1}{d}+\frac{2}{dm}$ for each $m$, and we get infinitely many walls near $\frac{1}{d}$.
\end{exa}

In these two examples, the walls accumulate from above. One may ask whether the walls also accumulate from below (equivalently, whether the walls satisfy the ascending chain condition, or ACC). The next example shows that this can also happen.

\begin{exa}
If $X$ is a klt surface with ample canonical divisor, and there exists a divisor $E$ over $X$ such that $a(E,X,0)<0$, then there is a wall at $-a(E,X,0)$. This is because  there exists a birational morphism $\pi\colon Y\to X$ that extracts $E$ as the unique exceptional divisor, and $K_Y+(-a(E,X,0)-\varepsilon)E = \pi^*K_X-\varepsilon E$ is ample when $0<\varepsilon\ll 1$. By \cite{Kol08-topo-BMY-ineq}*{43}, there is an unbounded set of klt surfaces with ample canonical divisor and volume $\le 2$. Moreover, the minimal log discrepancies of these surfaces can be arbitrarily close to zero (this follows from the general boundedness result in \cite{HMX14-ACC}, but can also be deduced through a direct computation). By our previous discussion, this implies that there are walls accumulating near $1$ (the accumulation is necessarily from below).
\end{exa}

\bibliographystyle{amsalpha}
\bibliography{biblio}

\end{document}